\documentclass{amsart}

\usepackage{mathptmx}

\usepackage[url=false, isbn=false, doi=false, backend=bibtex, firstinits=true, backref=true]{biblatex}
\usepackage{amssymb,amsthm,amsmath,amsfonts}
\usepackage{graphicx}
\usepackage[mathscr]{eucal}

\usepackage{bbm}

\usepackage{centernot}

\usepackage{xypic}

\allowdisplaybreaks
\allowbreak

\usepackage[pdfpagelabels,hypertexnames=true,plainpages=false,naturalnames=false]{hyperref}

\usepackage{color}
\definecolor{darkblue}{rgb}{0,0.1,0.5}
\hypersetup{colorlinks,linkcolor=darkblue,anchorcolor=darkblue,citecolor=darkblue}

\newtheorem{theorem}{Theorem}[section]
\newtheorem{lemma}[theorem]{Lemma}
\newtheorem{proposition}[theorem]{Proposition}
\newtheorem{corollary}{Corollary}

\newtheorem{definition}{Definition}[section]

\newtheorem{example}{Example}[section]

\newtheorem{remark}{Remark}
\newtheorem{note}{Note}

\numberwithin{equation}{section}

\newcommand{\N}{\mathbb{N}}
\newcommand{\Z}{\mathbb{Z}}
\newcommand{\R}{\mathbb{R}}

\newcommand{\Leng}{\mathcal{L}}

\newcommand{\supp}{\mathrm{supp}}

\newcommand{\Path}{\mathrm{P}}

\newcommand{\neig}{\mathrm{N}}
\newcommand{\tree}{\mathbb{T}}

\newcommand{\complete}{\mathrm{K}}
\newcommand{\completeO}{\mathrm{K}^{\rotatebox[origin=c]{180}{$\circlearrowright$}}}

\newcommand{\Star}{\mathrm{S}}

\newcommand{\board}{\mathcal{G}}
\newcommand{\bVert}{\mathcal{V}}
\newcommand{\bEdg}{\mathcal{E}}

\newcommand{\const}{\mathrm{H}}
\newcommand{\cVert}{\mathrm{V}}
\newcommand{\cEdg}{\mathrm{E}}

\newcommand{\homspace}{\mathrm{Hom}(\board,\const)}
\newcommand{\Loops}{\mathrm{Loop}}

\newcommand{\dist}{\mathrm{dist}}

\newcommand{\Pos}{\mathcal{P}}
\newcommand{\Dom}{D}

\renewcommand{\min}{\mathrm{min}}
\renewcommand{\max}{\mathrm{max}}

\begin{document}

\title{Strong spatial mixing in homomorphism spaces}

\author{Raimundo Brice\~no}
\address{Raimundo Brice\~no\\Department of Mathematics\\The University of British Columbia\\1984 Mathematics Road\\Vancouver, B.C., V6T 1Z2\\Canada}
\email{raimundo@math.ubc.ca}

\author{Ronnie Pavlov}
\address{Ronnie Pavlov\\Department of Mathematics\\University of Denver\\2280 S. Vine St.\\Denver, CO 80208\\USA}
\email{rpavlov@du.edu}

\thanks{The second author acknowledges the support of NSF grant DMS-1500685.}



\maketitle

\begin{abstract}
Given a countable graph $\board$ and a finite graph $\const$, we consider $\homspace$ the set of graph homomorphisms from $\board$ to $\const$ and we study Gibbs measures supported on $\homspace$. We develop some sufficient and other necessary conditions on $\homspace$ for the existence of Gibbs specifications satisfying strong spatial mixing (with exponential decay rate). We relate this with previous work of Brightwell and Winkler, who showed that a graph $\const$ has a combinatorial property called dismantlability if and only if for every $\board$ of bounded degree, there exists a Gibbs specification with unique Gibbs measure. We strengthen their result by showing that this unique Gibbs measure can be chosen to have weak spatial mixing, but we also show that there exist dismantlable graphs for which no Gibbs measure has strong spatial mixing.
\end{abstract}

\tableofcontents

\section{Introduction}
\label{section1}

In the past decades, spatial mixing properties in spin systems have been of interest because of their many applications. The property known as \emph{weak spatial mixing (WSM)} is related with uniqueness of Gibbs measures on countable graphs and the absence of phase transitions. On the other hand, \emph{strong spatial mixing (SSM)}, which is a strengthening of WSM, has been connected with the existence of efficient approximation algorithms for thermodynamic quantities \cite{1-gamarnik,1-marcus,1-briceno}, FPTAS for counting problems which are \#P-hard \cite{1-bandyopadhyay,1-weitz,3-gamarnik} and mixing time of the Glauber dynamics in some particular systems \cite{1-jerrum,1-dyer}.

In \cite{1-brightwell}, Brightwell and Winkler did a complete study of the family of \emph{dismantlable graphs}, including several interesting alternative characterizations. Among the equivalences discussed in that work, many involved a countable graph $\board$ (the \emph{board}), a finite graph $\const$ (the \emph{constraint graph}, assumed to be dismantlable) and the set of all graph homomorphisms from $\board$ to $\const$, which we denote here by $\homspace$. We call such a set of graph homomorphisms a \emph{homomorphism space}. In this context, we should understand the set of vertices $\cVert(\const)$ as the set of spins in some spin system living on vertices of $\board$. The adjacencies given by the set of edges $\cEdg(\const)$ indicate the pairs of spins that are allowed to be next to each other in $\board$, and the edges that are missing can be seen as \emph{hard constraints} in our system (i.e. pair of spins that cannot be adjacent in $\board$). Examples of such systems are very common. If we consider $\board = \Z^2$ and $\const_\varphi$ with $\cVert(\const_\varphi) = \{0,1\}$ and $\cEdg(\const_\varphi)$ containing every edge but the loop connecting $1$ with itself, then $\mathrm{Hom}(\Z^2,\const_\varphi)$ represents the support of the well-known \emph{hard-square model}, i.e. the set of independent sets in $\Z^2$, the square lattice.

We are interested in \emph{combinatorial (or topological) mixing properties} that are satisfied by a homomorphism space $\homspace$, i.e. properties that allow us to ``glue'' together sets of spins in $\board$. For example, the homomorphism space $\mathrm{Hom}(\Z^2,\const)$, where $\cVert(\const) = \{0,1\}$ and $\const$ has a unique edge connecting $0$ with $1$, has only two elements, both checkerboard patterns of $0$s and $1$s. Then, this homomorphism space lacks good combinatorial mixing properties since, for example, it is not possible to ``glue'' two $0$s together which are separated by an odd distance horizontally or vertically. Note that this is not the case for $\mathrm{Hom}(\Z^2,\const_\varphi)$, where the only difference is that $\const_\varphi$ has in addition an edge connecting $0$ with itself. A gluing property which will be of particular interest is \emph{strong irreducibility}. In \cite{1-brightwell}, dismantlable graphs were characterized as the only graphs $\const$ such that $\homspace$ is strongly irreducible for every $\board$.

In addition, we can consider a \emph{n.n. interaction} $\Phi$, which is a function associating some ``energy'' to every vertex and edge of a constraint graph $\const$. From this we can construct a \emph{Gibbs $(\board,\const,\Phi)$-specification} $\pi$, which is an ensemble of probability measures supported in finite portions of $\board$. Specifications are a common framework for working with spin systems and defining \emph{Gibbs measures} $\mu$. From this point it is possible to start studying spatial mixing properties, which combine the geometry of $\board$, the structure of $\const$, and the distributions induced by $\Phi$. In \cite{1-brightwell}, dismantlable graphs were characterized as the only graphs $\const$ for which for every board $\board$ of bounded degree there exists a n.n. interaction $\Phi$ such that the Gibbs $(\board,\const,\Phi)$-specification $\pi$ has no phase transition (i.e. there is a unique Gibbs measure).

In this work we study the problem of existence of strong spatial mixing measures supported on homomorphism spaces. First, we extend the results of Brightwell and Winkler on uniqueness, by characterizing dismantlable graphs as the only graphs $\const$ for which for every board $\board$ of bounded degree there exists a n.n. interaction $\Phi$ such that the Gibbs $(\board,\const,\Phi)$-specification $\pi$ satisfies WSM (see Proposition \ref{dism-highRate}). Then we study strong spatial mixing on homomorphism spaces. We give sufficient conditions on $\const$ and $\homspace$ for the existence of Gibbs $(\board,\const,\Phi)$-specifications satisfying SSM. Since SSM implies WSM, a necessary condition for SSM to hold in every board $\board$ is that $\const$ is dismantlable. We exhibit examples showing that SSM is a strictly stronger property, in terms of combinatorial properties of $\const$ and $\homspace$, than WSM. In particular, there exist dismantlable graphs where SSM fails for some boards $\board$.

The paper is organized as follows. In Section \ref{section2} and Section \ref{section3}, we introduce the necessary background for studying homomorphism spaces and Gibbs specifications. In Section \ref{section4}, we introduce meaningful combinatorial properties for studying SSM and homomorphism spaces in general, where strong irreducibility and the \emph{topological strong spatial mixing property} of \cite{1-briceno} play a fundamental role. In Section \ref{section5}, we introduce the \emph{unique maximal configuration (UMC) property} on $\homspace$ and show that this property is sufficient for having a Gibbs specification satisfying SSM (and in some sense, with arbitrarily high decay rate of correlations). In Section \ref{section6}, we introduce a fairly general family of graphs $\const$, strictly contained in the family of dismantlable graphs, such that $\homspace$ satisfies the UMC property for every board $\board$ (and therefore, we can always find a Gibbs specification satisfying SSM supported on $\homspace$). In Section \ref{section7}, we provide a summary of relationships and implications among the properties studied. In Section \ref{section8}, we focus in the particular case where $\const$ is a (looped) tree $T$ and conclude that the properties on $T$ yielding WSM for some measure on $\mathrm{Hom}(\board,T)$ coincide with those yielding SSM. Finally, in Section \ref{section9}, we provide examples illustrating the qualitative difference between the combinatorial properties necessary for WSM and SSM to hold in spin systems.


\section{Definitions and preliminaries}
\label{section2}

\subsection{Graphs}

A \emph{graph} is an ordered pair $G = (V(G),E(G))$ (or just $G = (V,E)$), where $V$ is a countable set of elements called \emph{vertices}, and $E$ is contained in the set of unordered pairs $\left\{\{x,y\}: x,y \in V\right\}$, whose elements we call \emph{edges}. We denote $x \sim y$ (or $x \sim_G y$ if we want to emphasize the graph $G$) whenever $\{x,y\} \in E$, and we say that $x$ and $y$ are \emph{adjacent}, and that $x$ and $y$ are the \emph{ends} of the edge $\{x,y\}$. A vertex $x$ is said to have a \emph{loop} if $\{x,x\} \in E$. The \emph{set of looped vertices} of a graph $G$ will be denoted $\Loops(G) := \left\{x \in V: \{x,x\} \in E\right\}$. A graph will be called \emph{simple} if $\Loops(G) = \emptyset$ and \emph{finite} if $|G| < \infty$, where $|G|$ denotes the cardinality of $V(G)$.

Fix $n \in \N$. A \emph{path (of length $n$)} in a graph $G$ will be a finite sequence of distinct edges $\{x_0,x_1\},\{x_1,x_2\},\dots,\{x_{n-1},x_{n}\}$. A single vertex $x$ will be considered to be a path of length $0$. A \emph{cycle (of length $n$)} will be a path such that $x_0 = x_n$ (notice that a loop is a cycle). A vertex $y$ will be said to be \emph{reachable} from another vertex $x$ if there exists a path (of some length $n$) such that $x_0 = x$ and $x_n = y$. A graph will be said to be \emph{connected} if every vertex is reachable from any other different vertex, and a \emph{tree} if it is connected and has no cycles. A graph which is a tree plus possibly some loops, will be called a \emph{looped tree}.

For a vertex $x$, we define its \emph{neighbourhood} $\neig(x)$ as the set $\{y \in V: y \sim x\}$. A graph $G$ will be called \emph{locally finite} if $|\neig(x)| < \infty$, for every $x \in V$, and a locally finite graph will have \emph{bounded degree} if $\Delta(G) := \sup_{x \in V} |\neig(x)| < \infty$. In this case, we call $\Delta(G)$ the \emph{maximum degree} of $G$. Given $d \in \N$, a graph of bounded degree is \emph{$d$-regular} if $|\neig(x)| = d$, for all $x \in V$.

Given a graph $G = (V,E)$, we say that a graph $G' = (V',E')$ is a \emph{subgraph of $G$} if $V' \subseteq V$ and $E' \subseteq E$. For a subset of vertices $A \subseteq V$, we define the subgraph of $G$ \emph{induced by $A$} as $G[A] := (A,E[A])$, where $E[A] := \{\{x,y\} \in E: x,y \in A\}$. Given two disjoint sets of vertices  $A_1,A_2 \subseteq V$, we define $E[A_1:A_2] := \left\{\{x,y\} \in E: x \in A_1,y \in A_2\right\}$, i.e. the set of edges with one end in $A_1$ and the other end in $A_2$.

We will usually use the letters $u$, $v$, etc. for denoting vertices in a finite graph, and $x$, $y$, etc. in an infinite one.

\subsection{Boards and constraint graphs}

In this work, inspired by \cite{1-brightwell}, we will consider mainly two kinds of graphs:
\begin{enumerate}
\item a \emph{board} $\board = (\bVert,\bEdg)$: countable, simple, connected, locally finite graph with at least two vertices, and
\item a \emph{constraint graph} $\const = (\cVert,\cEdg)$: finite graph, where loops are allowed.
\end{enumerate}

Fix a board $\board = (\bVert,\bEdg)$. Then, for $x,y \in \bVert$, we can define a natural distance function
\begin{equation}
\dist(x,y) := \min\{n: \mbox{$\exists$ a path of length $n$ s.t.  $x = x_0$ and $x_n = y$}\}, 
\end{equation}
which can be extended to subsets $A,B \subseteq \bVert$ as $\dist(A,B) = \min_{x \in A, y \in B} \dist(x,y)$. We denote $A \Subset B$ whenever a finite set $A \subseteq \bVert$ is contained in an infinite set $B \subseteq \bVert$. When denoting subsets of $\bVert$ that are singletons, brackets will usually be omitted, e.g. $\dist(x,A)$ will be regarded to be the same as $\dist(\{x\},A)$. 

We define the \emph{boundary} of $A \subseteq \bVert$ as the set $\partial A := \left\{x \in \bVert: \dist(x,A) = 1\right\}$ (notice that if $x \in A$, then $\dist(x,A) = 0$), and the \emph{closure} of $A$ as $\overline{A} = A \cup \partial A$. Given $n \in \N$, we call $\neig_n(A) := \left\{x \in \bVert: \dist(x,A) \leq n\right\}$ the \emph{$n$-neighbourhood} of $A$ (notice that $\neig_0(A) = A$ and $\neig_1(x) = \neig(x) \cup \{x\}$).

\begin{example}
Given $d \in \N$, two boards are of special interest (see Figure \ref{boards}):
\begin{itemize}
\item The \emph{$d$-dimensional hypercubic lattice} $\Z^d = \left(\bVert(\Z^d),\bEdg(\Z^d)\right)$, which is the $d$-regular countable infinite graph, where
\begin{equation}
\begin{array}{ccc}
\bVert(\Z^d) = \Z^d,	&	\mbox{ and }	&	\bEdg(\Z^d) = \left\{\{x,y\}: x,y \in \Z^d, \|x-y\|=1\right\},
\end{array}
 \end{equation}
 with $\|x\| = \sum_{i=1}^{d}\left|x_i\right|$ the $1$-norm.
\item The \emph{$d$-regular tree} $\tree_d = \left(\bVert(\tree_d),\bEdg(\tree_d)\right)$, which is the unique simple graph that is a countable infinite $d$-regular tree. This board is also known as the \emph{Bethe lattice}.
\end{itemize}
\end{example}

\begin{figure}[ht]
\centering
\includegraphics[scale = 0.6]{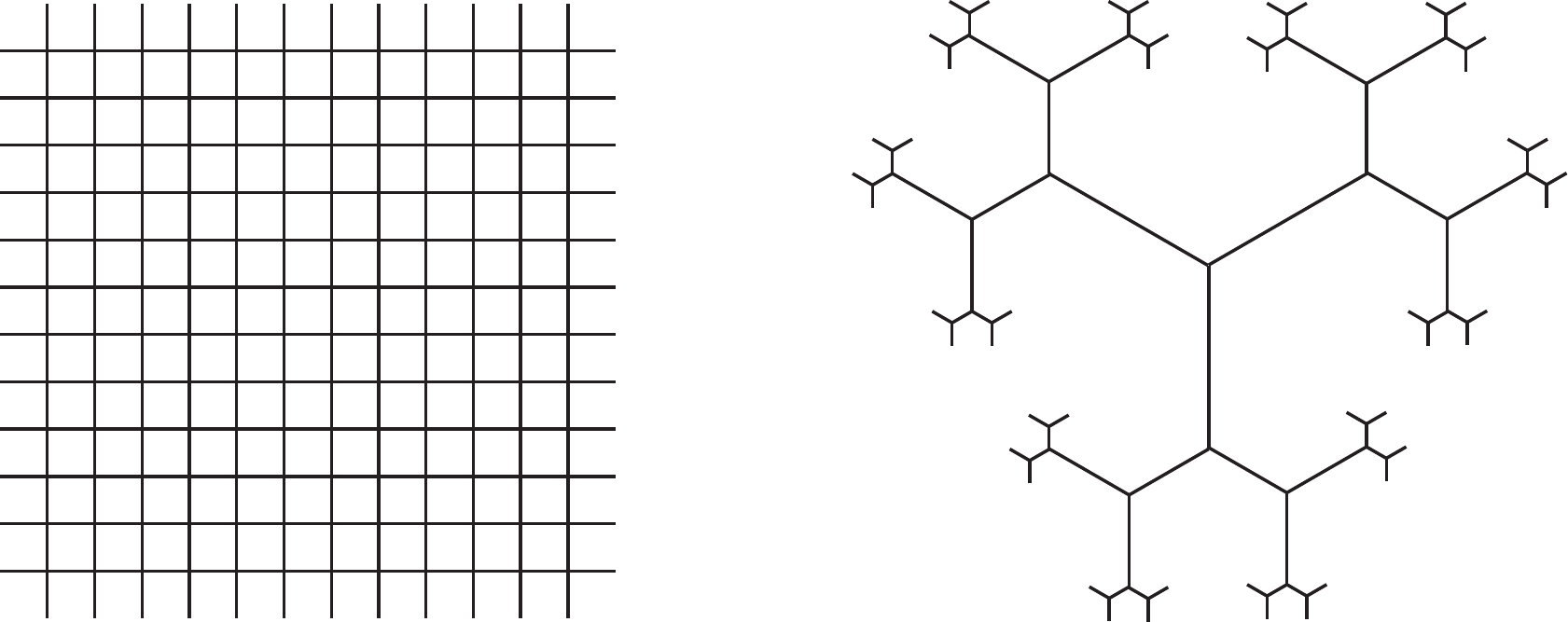} 
\caption{A sample of the boards $\Z^2$ and $\tree_3$.}
\label{boards}
\end{figure}

A difference between boards and constraint graphs is that the latter must be finite. Another one is that constraint graphs are allowed to have loops. Whenever we have a finite graph $G = (V,E)$, we will denote by $G^{\rotatebox[origin=c]{180}{$\circlearrowright$}} = (V,E^{\rotatebox[origin=c]{180}{$\circlearrowright$}})$ the graph obtained by adding loops to every vertex, i.e. $E^{\rotatebox[origin=c]{180}{$\circlearrowright$}} = E \cup \{\{x,x\}: x \in V\}$ and $\Loops(G^{\rotatebox[origin=c]{180}{$\circlearrowright$}}) = V$.

A finite graph will be called \emph{complete} if $x \sim y$ iff $x \neq y$. The complete graph with $n$ vertices will be denoted $\complete_n$ (notice that $\Loops(\complete_n) = \emptyset$). A finite graph will be called \emph{loop-complete} if $x \sim y$, for every $x,y$. Notice that the loop-complete graph with $n$ vertices is $\completeO_n$. The graphs $\complete_n$ and $\completeO_n$ are very important examples of constraint graphs, which relate to proper colourings of boards and unconstrained models, respectively (see Example \ref{models}).

\begin{figure}[ht]
\centering
\includegraphics[scale = 0.5]{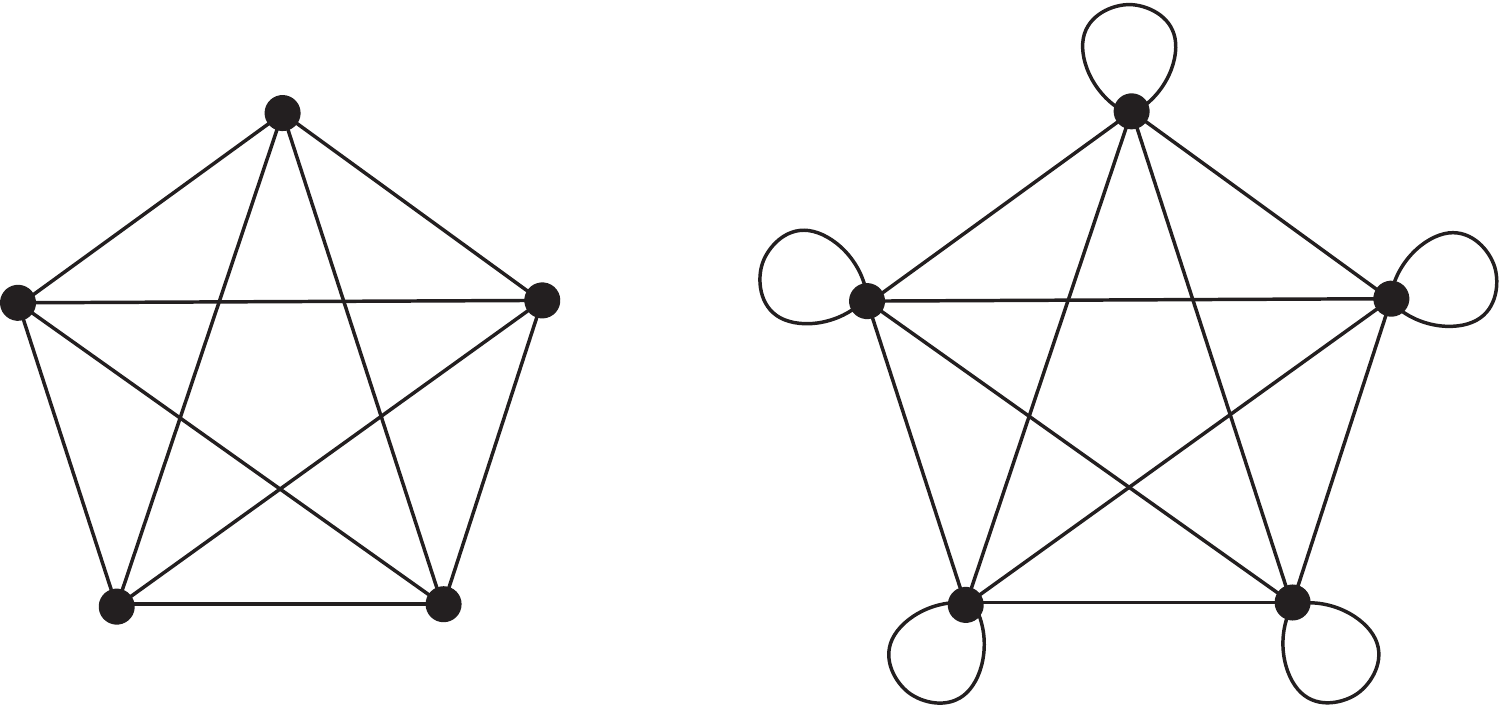} 
\caption{The graphs $\complete_n$ and $\mathrm{K}^{\protect\rotatebox[origin=c]{180}{$\circlearrowright$}}_n$, for $n = 5$.}
\label{complete}
\end{figure}

Other relevant examples are the following.

\begin{example}
The constraint graph given by:
\begin{equation}
\const_{\varphi} := \left(\{0,1\},\left\{\{0,0\},\{0,1\}\right\}\right),
\end{equation}
shown in Figure \ref{hard}, is related to the \emph{hard-core model} (see Example \ref{models}).

\begin{figure}[ht]
\centering
\includegraphics[scale = 0.5]{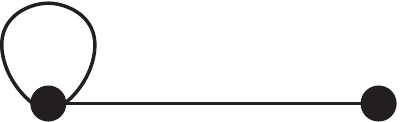} 
\caption{The graph $\const_{\varphi}$.}
\label{hard}
\end{figure}

Another one is, given $n \in \N$, the \emph{$n$-star graph}
\begin{equation}
\Star_n = \left(\{0,1,\dots,n\},\left\{\{0,1\},\dots,\{0,n\}\right\}\right).
\end{equation}

In addition, it will be useful to consider the graphs
\begin{equation}
\Star^{\rm o}_n = \left(V(\Star_n),E(\Star_n) \cup \{\{0,0\}\}\right)
\end{equation}
and $\Star^{\rotatebox[origin=c]{180}{$\circlearrowright$}}$ (see Figure \ref{stars}). Notice that $\const_{\varphi} = \Star^{\rm o}_1$.

\begin{figure}[ht]
\centering
\includegraphics[scale = 0.5]{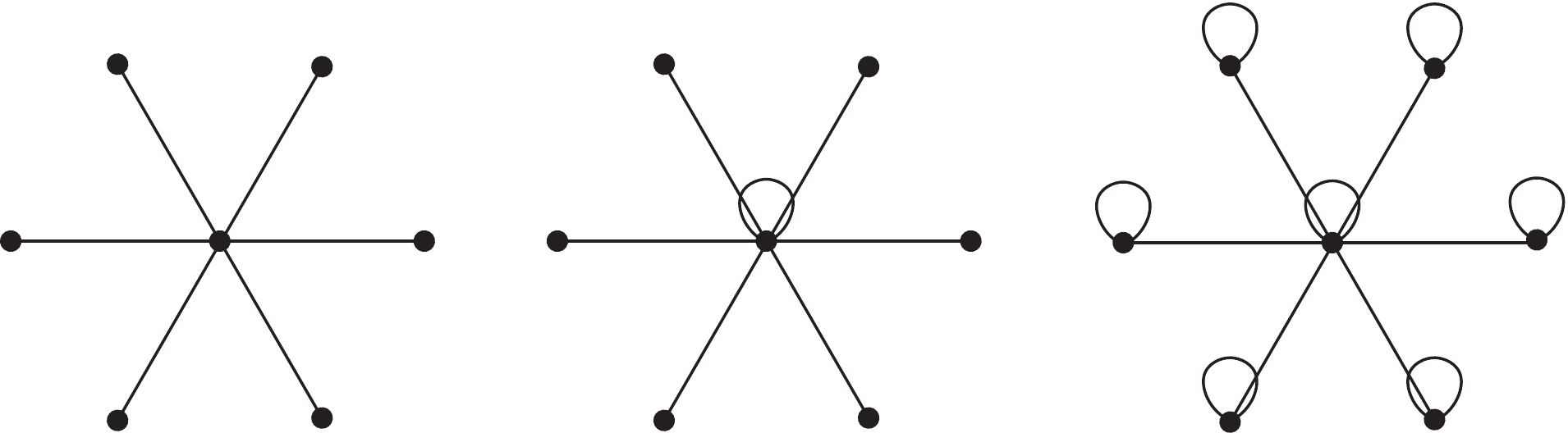} 
\caption{The graphs $\Star_6$, $\Star^{\rm o}_6$ and $\Star^{\protect\rotatebox[origin=c]{180}{$\circlearrowright$}}_6$.}
\label{stars}
\end{figure}
\end{example}

\subsection{Homomorphism spaces}

In this work we relate boards and constraint graphs via \emph{graph homomorphisms}. A graph homomorphism $\alpha: G_1 \to G_2$ from a graph $G_1 = (V_1,E_1)$ to a graph $G_2 = (V_2,E_2)$ is a mapping $\alpha: V_1 \to V_2$ such that
\begin{equation}
\{x,y\} \in E_1 \implies \{\alpha(x),\alpha(y)\} \in E_2.
\end{equation}

Given two graphs $G_1$ and $G_2$, we will denote by $\mathrm{Hom}(G_1,G_2)$ the set of all graph homomorphisms $\alpha: G_1 \to G_2$, from $G_1$ to $G_2$.

\subsubsection{Homomorphisms as configurations}

Fix a board $\board = (\bVert,\bEdg)$ and a constraint graph $\const = (\cVert,\cEdg)$. We will call the set $\homspace$ a \emph{homomorphism space}. In this context, the graph homomorphisms that belong to $\homspace$ will be called \emph{points} and denoted with the Greek letters $\omega$, $\upsilon$, etc. Notice that a point $\omega \in \homspace$ can be understood as a ``colouring'' of $\bVert$ with elements from $\cVert$ such that $x \sim_\board y \implies \omega(x) \sim_\const \omega(y)$. In other words, $\omega$ is a colouring of $\board$ that respects the constraints imposed by $\const$ with respect to adjacency.

\begin{example}
For $d \in \N$, two examples of homomorphism spaces are:
\begin{itemize}
\item $\mathrm{Hom}(\Z^d,\const_{\varphi})$, i.e. the set of elements in $\{0,1\}^{\Z^d}$ with no adjacent $1$s, and
\item	$\mathrm{Hom}(\tree_d,\complete_q)$, with $d \geq q$, i.e. the set of proper $q$-colourings of the $d$-regular tree.
\end{itemize}
\end{example}

Given $A \subseteq \bVert$, a \emph{configuration} will be any map $\alpha: A \to \cVert$ (i.e. $\alpha \in \cVert^A$), which will usually be denoted with the Greek letters $\alpha$, $\beta$, etc. The set $A$ is called the \emph{shape} of $\alpha$, and a configuration will be said to be \emph{finite} if its shape is finite. For any configuration $\alpha$ with shape $A$ and $B \subseteq A$, $\left.\alpha\right|_B$ denotes the restriction of $\alpha$ to $B$, i.e. the map from $B$ to $\cVert$ obtained by restricting the domain of $\alpha$ to $B$. For $A$ and $B$ disjoint sets, $\alpha \in \cVert^A$ and $\beta \in \cVert^B$, $\alpha\beta$ will be the configuration on $A \cup B$ defined by $\left.(\alpha\beta)\right|_A = \alpha$ and $\left.(\alpha\beta)\right|_B = \beta$. Notice that a point is a configuration with shape $\bVert$.

Given two configurations $\alpha_1, \alpha_2 \in \cVert^A$ and $B \subseteq A$, we define their \emph{set of $B$-di\-sa\-gree\-ment} as
\begin{equation}
\Sigma_B(\alpha_1,\alpha_2) := \left\{x \in B: \alpha_1(x) \neq \alpha_2(x) \right\}.
\end{equation}

\subsubsection{Locally/globally admissible configurations}

Fix a homomorphism space $\homspace$ and a set $A \subseteq \bVert$. A configuration $\alpha \in \cVert^A$ is said to be \emph{globally admissible} if there exists $\omega \in \homspace$ such that $\left.\omega\right|_A = \alpha$. A configuration $\alpha \in \cVert^A$ is said to be \emph{locally admissible} if $\alpha$ is a graph homomorphism from $\board[A]$ to $\const$, i.e. if $\alpha \in \mathrm{Hom}(\board[A],\const)$. A globally admissible configuration is also locally admissible, but the converse is false. In addition, notice that if a configuration $\alpha \in \cVert^A$ is globally (resp. locally) admissible, then $\left.\alpha\right|_B$ is also globally (resp. locally) admissible, for any $B \subseteq A$.

The \emph{language} $\Leng(\homspace)$ of a homomorphism space $\homspace$ is the set of all finite globally admissible configurations, i.e.
\begin{equation}
\Leng(\homspace) := \bigcup_{A \Subset \bVert} \Leng_A(\homspace),
\end{equation}
where $\Leng_A(\homspace) := \left\{\left.\omega\right|_A: \omega \in \homspace \right\}$, for $A \subseteq \bVert$. 

Given $A \subseteq \bVert$ and a configuration $\alpha \in \cVert^A$, we define the \emph{cylinder set $[\alpha]^\board_\const$} as
\begin{equation}
[\alpha]^\board_\const := \left\{\omega \in \homspace: \left.\omega\right|_A = \alpha\right\}.
\end{equation}

Notice that $\alpha \in \cVert^A$ is globally admissible iff $\alpha \in \Leng_A(\homspace)$ iff $[\alpha]^\board_\const \neq \emptyset$.


\section{Gibbs measures}
\label{section3}

\subsection{Constrained interactions and Gibbs specifications}

Given a constraint graph $\const$, a \emph{nearest-neighbour (n.n.) interaction $\Phi$} for $\const$ will be any function $\Phi: \cVert \cup \cEdg \to (-\infty,0]$. We will call the pair $(\const,\Phi)$ a \emph{constrained n.n. interaction}.

\begin{example}
\label{models}
Let $q \in \N$ and $\beta > 0$. Many constrained n.n. interactions represent well-known classical models. (In all of the following models, the parameter $\beta$ is classically referred to as the inverse temperature.)
\begin{itemize}
\item Ferromagnetic Potts $(\completeO_q,\beta\Phi^{\rm FP})$: $\left.\Phi^{\rm FP}\right|_{\cVert} \equiv 0$, $\Phi^{\rm FP}(\{u,v\}) = -\mathbbm{1}_{\{u = v\}}$.
\item Anti-ferromagnetic Potts $(\completeO_q,\beta\Phi^{\rm AP})$: $\left.\Phi^{\rm AP}\right|_{\cVert} \equiv 0$, $\Phi^{\rm AP}(\{u,v\}) = -\mathbbm{1}_{\{u \neq v\}}$.
\item Proper $q$-colourings $(\complete_q, \Phi^{\rm PC})$: $\left.\Phi^{\rm PC}\right|_{\cVert \cup \cEdg} \equiv 0$.
\item Hard-core $(\const_{\varphi},\beta\Phi_{\rm HC})$: $\Phi^{\rm HC}(0) = 0$, $\Phi^{\rm HC}(1) = -1$, $\left.\Phi^{\rm HC}\right|_{\cEdg} \equiv 0$.
\item Multi-type Widom-Rowlinson $(\mathrm{S}^{\rotatebox[origin=c]{180}{$\circlearrowright$}}_q,\beta\Phi^{\rm WR})$: $\Phi^{\rm WR}(v) = -\mathbbm{1}_{\{v \neq 0\}}$, $\left.\Phi^{\rm WR}\right|_{\cEdg} \equiv 0$.
\end{itemize}
\end{example}

Now, given a board $\board$ and a constrained n.n. interaction $(\const,\Phi)$, for any set $A \Subset \bVert$ and $\omega \in \homspace$, we define the \emph{energy function}
\begin{equation}
E^\Phi_{A,\omega}: \{\alpha \in \cVert^A: \alpha\left.\omega\right|_{A^c} \in \homspace\} \to \R
\end{equation}
as
\begin{align}
E^\Phi_{A,\omega}(\alpha) := \sum_{x \in A}\left(\Phi(\alpha(x)) + \frac{1}{2}\sum_{y \in A: y \sim x} \Phi(\alpha(x),\alpha(y)) + \sum_{y \in \partial A: y \sim x} \Phi(\alpha(x),\omega(y))\right).
\end{align}

Then, given $A \Subset \bVert$ and $\omega \in \homspace$, we can define a probability measure on $\Leng_A(\homspace)$ given by
\begin{equation}
\pi^\omega_A(\alpha) :=
\begin{cases}
\frac{1}{\mathrm{Z}^\Phi_{A,\omega}}e^{-E^\Phi_{A,\omega}(\alpha)}		&	\mbox{if } \alpha\left.\omega\right|_{A^c} \in \homspace,	\\
0														&	\mbox{otherwise,}
\end{cases}
\end{equation}
where
\begin{equation}
\mathrm{Z}^\Phi_{A,\omega} := \sum_{\alpha:  \alpha\left.\omega\right|_{A^c} \in \homspace} e^{-E^\Phi_{A,\omega}(\alpha)}
\end{equation}
is called the \emph{partition function}. For $B \subseteq A$ and $\beta \in \cVert^B$, we marginalize as follows:
\begin{equation}
\pi^\omega_A(\beta) = \sum_{\alpha \in \Leng_A(\homspace): \left.\alpha\right|_B = \beta} \pi^\omega_A(\alpha).
\end{equation}

The collection $\pi = \left\{\pi^\omega_A: A \Subset \bVert, \omega \in \homspace\right\}$ will be called \emph{Gibbs $(\board,\const,\Phi)$-specification}. If we take $\Phi \equiv 0$, then we call the Gibbs $(\board,\const,0)$-specification $\pi$, the \emph{uniform Gibbs specification} on $\homspace$ (see the case of proper $q$-colourings in Example \ref{models}).

\subsection{Gibbs measures}

A Gibbs $(\board,\const,\Phi)$-specification is regarded as a meaningful representation of an ideal physical situation where every finite volume $A$ in the space is in thermodynamical equilibrium the exterior. The extension of this idea to infinite volumes is via a particular class of probability measures on $\homspace$ called \emph{Gibbs measures}.

\subsubsection{Borel probability measures and Markov random fields}

Given a homomorphism space $\homspace$ and $A \subseteq \bVert$, we denote by $\mathcal{F}_A$ the $\sigma$-algebra generated by all the cylinder sets $[\alpha]^\board_\const$ with shape $A$, and we equip $\cVert^\bVert$ with the $\sigma$-algebra $\mathcal{F} = \mathcal{F}_\bVert$. A \emph{Borel probability measure} $\mu$ on $\cVert^\bVert$ is a measure such that $\mu(\cVert^\bVert) = 1$, determined by its values on cylinder sets of finite configurations. Given a cylinder set $[\alpha]^\board_\const$ and a measure $\mu$, we will just write $\mu(\alpha)$ for the value of $\mu([\alpha]^\board_\const)$, whenever $\board$ and $\const$ are understood. The \emph{support} of such a measure $\mu$ will be defined as
\begin{equation}
\supp(\mu) := \left\{\omega \in \homspace: \mu(\left.\omega\right|_A) > 0, \mbox{ for all } A \Subset \bVert\right\}.
\end{equation}

We will denote by $\mathcal{M}_{1}(\homspace)$ the set of all Borel probability measures whose support $\supp(\mu)$ is contained in $\homspace$.

\begin{definition}
A measure $\mu \in \mathcal{M}_{1}(\homspace)$ is a \emph{Markov random field (MRF)} if, for any subset $A \Subset \bVert$, any $\alpha \in \cVert^A$, any $B \Subset \bVert$ such that $\partial A \subseteq B \subseteq \bVert \setminus A$, and any $\beta \in \cVert^B$ with $\mu(\beta) > 0$, it is the case that
\begin{equation}
\mu\left(\alpha \middle\vert \beta\right) = \mu\left(\alpha \middle\vert \left.\beta\right|_{\partial A}\right).
\end{equation}
\end{definition}

In other words, an MRF is a measure where every finite configuration conditioned to its boundary is independent of the configuration on the complement.

\subsubsection{Nearest-neighbour Gibbs measures}

\begin{definition}
A \emph{nearest-neighbour (n.n.) Gibbs measure} for a Gibbs $(\board,\const,\Phi)$-spe\-ci\-fi\-ca\-tion $\pi$ is a measure $\mu \in \mathcal{M}_{1}(\homspace)$ such that for any $A \Subset \board$ and $\omega \in \homspace$ with $\mu(\left.\omega\right|_{\partial A}) > 0$, we have $\mathrm{Z}^\Phi_{A,\omega} > 0$ and
\begin{equation}
\label{DLR}
\begin{array}{cc}
\mathbb{E}_\mu(\mathbbm{1}_{[\alpha]^\board_\const} \vert \mathcal{F}_{A^c})(\omega) = \pi^\omega_A(\alpha)	&	 \mbox{ $\mu$-a.s.},
\end{array}
\end{equation}
for every $\alpha \in \Leng_A(\homspace)$.
\end{definition}

Notice that every n.n. Gibbs measure is an MRF because the formula for $\pi^\omega_A$ only depends on $\left.\omega\right|_{\partial A}$.

If $\homspace \neq \emptyset$, every Gibbs $(\board,\const,\Phi)$-specification $\pi$ has at least one n.n. Gibbs measure (special case of a result in \cite{1-dobrushin}, see also \cite{2-brightwell}). Often there are multiple n.n. Gibbs measures for a single $\pi$. This phenomenon is usually called a \emph{phase transition}. There are several conditions that guarantee uniqueness of n.n. Gibbs measures. Some of them fall into the category of \emph{spatial mixing} properties, introduced in the next section.


\section{Properties of a Gibbs $(\board,\const,\Phi)$-specification $\pi$}
\label{section4}

One of our main purposes in this work is to understand the combinatorial properties that constraint graphs $\const$ and homomorphism spaces $\homspace$ should satisfy in order to admit the existence of a Gibbs $(\board,\const,\Phi)$-specification $\pi$ with some specific measure-theoretical properties, here called \emph{spatial mixing} properties.

\subsection{Spatial mixing properties of $\pi$}

In the following, let $f:\N \rightarrow \R_{\geq 0}$ be a function such that $f(n) \searrow 0$ as $n \to \infty$, that will be referred as a \emph{decay function}. We will loosely use the term ``spatial mixing property'' to refer to any measure-theoretical property satisfied by $\pi$ defined via a decay of correlation of events (or configurations) with respect to the distance that separates the shapes where they are supported.

The first property introduced here, \emph{weak spatial mixing (WSM)}, has direct connections with the nonexistence of phase transitions and has been studied in several works, explicitly and implicitly (see \cite{1-brightwell,2-weitz}). The next one, \emph{strong spatial mixing (SSM)}, is a strengthening of WSM that also has connections with meaningful physical idealizations (see \cite{1-martinelli}) and has also proven to be useful for developing approximation algorithms (see \cite{1-weitz}). The constrained n.n. interactions with a unique Gibbs measure have been already studied and, to some extent, characterized (see the work of Brightwell and Winkler on \emph{dismantlable graphs} \cite{1-brightwell}); we will show later (see Theorem \ref{dism-charact}) that their proof also gives WSM of the Gibbs specification. The main aim of this work is to develop a somewhat analogous framework and sufficiently general conditions under which constrained n.n. interactions yield specifications which satisfy SSM.

\begin{definition}
A Gibbs $(\board,\const,\Phi)$-specification $\pi$ satisfies \emph{weak spatial mixing (WSM) with rate $f$} if for any $A \Subset \bVert$, $B \subseteq  A$, $\beta \in \cVert^B$ and $\omega_1,\omega_2 \in \homspace$,
\begin{equation}
\left| \pi_A^{\omega_1}(\beta) - \pi_A^{\omega_2}(\beta) \right| \leq \left|B\right|f(\dist(B,\partial A)).
\end{equation}
\end{definition}

We use the convention that $\dist(B,\emptyset) = \infty$. Considering this, we have the following definition, a priori stronger than WSM.

\begin{definition}
\label{SSMspec}
A Gibbs $(\board,\const,\Phi)$-specification $\pi$ satisfies \emph{strong spatial mixing (SSM) with rate $f$} if for any $A \Subset \bVert$, $B \subseteq  A$, $\beta \in \cVert^B$ and $\omega_1,\omega_2 \in \homspace$,
\begin{equation}
\left| \pi_A^{\omega_1}(\beta) - \pi_A^{\omega_2}(\beta) \right| \leq |B|f\left(\dist(B,\Sigma_{\partial A}(\omega_1,\omega_2))\right).
\end{equation}
\end{definition}

Notice that $\dist(B,\Sigma_{\partial A}(\omega_1,\omega_2)) \geq \dist(B,\partial A)$.

We will say that a specification $\pi$ satisfies WSM (resp. SSM) if it satisfies WSM (resp. SSM) with rate $f$, for some $f$ as before. For $\gamma > 0$, we will say that a specification $\pi$ satisfies \emph{exponential WSM} (resp. \emph{exponential SSM}) with decay rate $\gamma$ if it satisfies WSM (resp. SSM) with decay function $f(n) = Ce^{-\gamma n}$ for some $C > 0$.

\begin{lemma}[{\cite[Lemma 2.3]{2-marcus}}]
\label{ssmSing}
Let $\pi$ be a Gibbs $(\board,\const,\Phi)$-specification such that for any $A \Subset \bVert$, $x \in A$, $\beta \in \cVert^{\{x\}}$ and $\omega_1,\omega_2 \in \homspace$,
\begin{equation}
\left| \pi_A^{\omega_1}(\beta) - \pi_A^{\omega_2}(\beta) \right| \leq f\left(\dist(x,\Sigma_{\partial A}(\omega_1,\omega_2))\right).
\end{equation}

Then, $\pi$ satisfies SSM with rate $f$.
\end{lemma}

\begin{remark}
The proof of Lemma \ref{ssmSing} given in \cite{2-marcus} is for MRFs $\mu$ satisfying exponential SSM with $\board = \Z^d$, but its generalization to specifications and more general boards is direct. We don't know if there is an analogous lemma for WSM.
\end{remark}

If a Gibbs $(\board,\const,\Phi)$-specification $\pi$ satisfies WSM, then there is a unique n.n. Gibbs measure $\mu$ for $\pi$ (see \cite{2-weitz}).

\begin{example}
There are some well-known Gibbs specifications that satisfy exponential SSM (and therefore, WSM). Recall the constrained n.n. interactions introduced in Example~\ref{models}.
\begin{itemize}
\item $(\Z^2,\completeO_q,\beta\Phi^{\rm FP})$, for any $q \in \N$ and small enough $\beta$ (see \cite{1-adams,1-berg}).
\item $(\Z^2,\completeO_q,\beta\Phi^{\rm AP})$ for $q \geq 6$ and any $\beta > 0$ (see \cite{2-goldberg}).
\item $(\tree_d,\complete_q,0)$, for $q \geq 1 + \delta^* d$, where $\delta^* = 1.763\dots$ is the unique solution to $xe^{-1/x} = 1$ (see \cite{1-ge,1-goldberg}).
\item $(\board,\const_\varphi,\beta\Phi^{\rm HC})$, for any $\board$ and $\beta$ such that $\Delta(\board) \leq d$ and $e^{\beta} < \lambda_{\rm c}(d) := \frac{(d-1)^{(d-1)}}{(d-2)^d}$ (see \cite{1-weitz}).
\item $(\Z^d,\mathrm{S}^{\rotatebox[origin=c]{180}{$\circlearrowright$}}_q,\beta\Phi^{\rm WR})$, for any $q \in \N$ and small enough $\beta$ (see \cite{1-adams,1-berg}).
\end{itemize}
\end{example}

There are more general sufficient conditions for having exponential SSM (for instance, see the discussion in \cite{2-marcus}).

\subsection{Graph-theoretical properties of $\const$}

Here we introduce some structural properties concerning constraint graphs, which will later be shown to have various implications for homomorphism spaces $\homspace$ and Gibbs $(\board,\const,\Phi)$-specifications.

The first property is the existence of a special vertex which is adjacent to every other vertex (including itself).

\begin{definition}
Given a constraint graph $\const$, we say that $s \in \cVert$ is a \emph{safe symbol} if $\{s,v\} \in \cEdg$, for every $v \in \cVert$.
\end{definition}

\begin{example}
The constraint graph $\const_{\varphi}$ has a safe symbol (see Figure \ref{hard}).
\end{example}

The next definition is a structural description of a class of graphs introduced in \cite{1-nowakowski} and heavily studied and characterized in \cite{1-brightwell}.

\begin{definition}
Given a constraint graph $\const$ and $u,v \in \cVert$ such that $\neig(u) \subseteq \neig(v)$, a \emph{fold} is a homomorphism $\alpha: \const \to \const[\cVert \setminus \{u\}]$ such that $\alpha(u) = v$ and $\left.\alpha\right|_{\cVert \setminus \{u\}} = \left.\mathrm{id}\right|_{\cVert \setminus \{u\}}$.

A constraint graph $\const$ is \emph{dismantlable} if there is a sequence of folds reducing $\const$ to a graph with a single vertex (with or without a loop).
\end{definition}

Notice that a fold $\alpha: \const \to \const[\cVert \setminus \{u\}]$ amounts to just removing $u$ and edges containing it from the graph 
$\const$, as long as a suitable vertex $v$ exists which can ``absorb'' $u$.

The following proposition is a good example of the kind of results that we aim to achieve.

\begin{proposition}
Let $\const$ be a constraint graph. Then, $\const$ is dismantlable iff for every board $\board$ of bounded degree, there exists a n.n. interaction $\Phi$ such that the Gibbs $(\board,\const,\Phi)$-specification satisfies exponential WSM with arbitrarily high decay rate.
\end{proposition}

See Proposition \ref{dism-charact} for a proof of this result, as we remark there, it is essentially due to Brightwell and Winkler \cite{1-brightwell}. One of our goals is to prove similar statements in which ``WSM'' is replaced by ``SSM.''

\subsection{Combinatorial properties of $\homspace$}

\begin{definition}
A homomorphism space $\homspace$ is said to be \emph{strongly irreducible with gap $g \in \N$} if for any pair of nonempty (disjoint) finite subsets $A,B \Subset \bVert$ such that $\dist(A,B) \geq g$, and for every $\alpha \in \cVert^A$, $\beta \in \cVert^B$,
\begin{equation}
[\alpha]^\board_\const, [\beta]^\board_\const \neq \emptyset \implies [\alpha\beta]^\board_\const \neq \emptyset.
\end{equation}
\end{definition}

\begin{remark}
Since a homomorphism space is a compact space, it does not make a difference if the shapes of $A$ and $B$ are allowed to be infinite in the definition of strong irreducibility.
\end{remark}

Now we proceed to adapt to the context of homomorphism spaces a combinatorial property that was originally introduced in \cite{1-briceno} in the context of \emph{$\Z^d$ shift spaces}. This condition was used in \cite{1-briceno} to give a partial characterization of systems in $\board = \Z^d$ that admit n.n. Gibbs measures satisfying SSM, with special emphasis in the case $d=2$.

\begin{definition}
A homomorphism space $\homspace$ is \emph{topologically strong spatial mixing (TSSM) with gap $g \in \N$}, if for any $A,B,S \Subset \bVert$ such that $\dist(A,B) \geq g$, and for every $\alpha \in \cVert^A$, $\beta \in \cVert^B$ and $\sigma \in \cVert^S$,
\begin{equation}
[\alpha \sigma]^\board_\const,[\sigma \beta]^\board_\const \neq \emptyset \implies [\alpha \sigma \beta]^\board_\const \neq \emptyset.
\end{equation}
\end{definition}

Note that TSSM with gap $g$ implies strong irreducibility with gap $g$ (by taking $S = \emptyset$). Clearly, strong irreducibility (resp. TSSM) with gap $g$ implies strong irreducibility (resp. TSSM) with gap $g+1$. We will say that a shift space satisfies strong irreducibility (resp. TSSM) if it satisfies strong irreducibility (resp. TSSM) with gap $g$, for some $g \in \N$.

A useful tool when dealing with TSSM is the next lemma, which states that if we have the TSSM property for single vertices, then we have it uniformly (in terms of separation distance) for any pair of finite sets $A$ and $B$.

\begin{lemma}[{\cite{1-briceno}}]
\label{lemSing}
Suppose that $g \in \N$ and $\homspace$ is a homomorphism space such that for every pair of sites $x,y \in \bVert$ with $\dist(x,y) \geq g$, $S \Subset \bVert$, and $\alpha \in \cVert^{\{x\}}$, $\beta \in \cVert^{\{y\}}$, $\sigma \in \cVert^{S}$ with $[\alpha\sigma]^\board_\const, [\sigma\beta]^\board_\const \neq \emptyset$, it is the case that $[\alpha\sigma\beta]^\board_\const \neq \emptyset$. Then, $\homspace$ satisfies TSSM with gap $g$.
\end{lemma}

The proof of Lemma \ref{lemSing} (which is a proof by induction on $|A| + |B|$) can be found in \cite{1-briceno} for the case $\board = \Z^d$, but the generalization is straightforward.

The following property was introduced in \cite{1-marcus} in the context of $\Z^d$ shift spaces, and here we proceed to adapt it to the case of homomorphism spaces.

\begin{definition}
\label{ssf}
A homomorphism space $\homspace$ is \emph{single-site fillable (SSF)} if for every site $x \in \bVert$ and $B \subseteq \partial\{x\}$, any graph homomorphism $\beta: \board[B] \to \const$ can be extended to a graph homomorphism $\alpha: \board[B \cup \{x\}] \to \const$ (i.e. $\alpha$ is such that $\left.\alpha\right|_{B} = \beta$).
\end{definition}

\begin{note}
Notice that a homomorphism space $\homspace$ satisfies SSF iff every locally admissible configuration is globally admissible (see \cite{1-marcus}).
\end{note}

\begin{example}
\label{exmpC4}
The homomorphism space $\mathrm{Hom}(\Z^d,\complete_q)$ satisfies SSF iff $q \geq 2d+1$.
\end{example}

\subsection{Relationships among properties}

Sometimes, when a property of $\homspace$ is satisfied for every board $\board$, one can conclude facts about $\const$. A simple example is the following.

\begin{proposition}
\label{safessf}
Let $\const$ be a constraint graph such that $\homspace$ satisfies SSF, for every board $\board$. Then $\const$ has a safe symbol.
\end{proposition}

\begin{proof}
Let $\board = \Star_{|\const|}$ (the $n$-star graph with $n = |\const|$) and let $x \in \bVert$ to be the central vertex with boundary $\partial\{x\} = \left\{y_1,\dots,y_{|\cVert|}\right\}$. Write $\cVert = \{v_1,\cdots,v_{|\cVert|} \}$ and take $\beta \in \cVert^{\partial \{x\}}$ such that $\beta(y_i) = v_i$, for $1 \leq i \leq |\cVert|$. Then, by SSF, there exists a graph homomorphism $\alpha: \board[\partial \{x\} \cup \{x\}] \to \const$ such that $\left.\alpha\right|_{\partial \{x\}} = \beta$. Since $\alpha$ is a graph homomorphism, $x \sim_\board y_i \implies \alpha(x) \sim_\const \alpha(y_i)$. Therefore, $\alpha(x) \sim_\const v_i$, for every $i$, so $\alpha(x)$ is a safe symbol for $\const$.
\end{proof}

Notice that the converse also holds, i.e. if a constraint graph $\const$ has a safe symbol, then $\homspace$ satisfies SSF, for every board $\board$. 

\begin{proposition}
\label{ssftssm}
If $\homspace$ satisfies SSF, then it satisfies TSSM with gap $g = 2$.
\end{proposition}

\begin{proof}
Since $\homspace$ satisfies SSF, every locally admissible configuration is globally admissible. If we take $g = 2$, for all disjoint sets $A,S,B \Subset \bVert$ such that $\dist(A,B) \geq g$ and for every $\alpha \in \cVert^A$, $s \in \cVert^S$ and $\beta \in \cVert^B$, if $[\alpha\sigma]^\board_\const, [\sigma\beta]^\board_\const \neq \emptyset$, in particular we have $\alpha\sigma$ and $\sigma\beta$ are locally admissible. Since $\dist(A,B) \geq g = 2$, $A$ and $B$ contain no adjacent vertices, and so $\alpha\sigma\beta$ must be locally admissible, too. Then, by SSF, $\alpha\sigma\beta$ is globally admissible and, therefore, $[\alpha\sigma\beta]^\board_\const \neq \emptyset$.
\end{proof}

Summarizing, given a homomorphism space $\homspace$, we have the following implications:
\begin{align}
\const \mbox{ has a safe symbol}	&	\implies	\homspace \mbox{ satisfies SSF}		\\
							&	\implies	\homspace \mbox{ satisfies TSSM}		\\
							&	\implies	\homspace \mbox{ is strongly irreducible}, \label{TSSM-irred}
\end{align}
and all implications are strict in general (even if we fix $\board$ to be a particular board, for example $\board = \Z^2$). See \cite{1-marcus,1-briceno} for examples that illustrate the differences among some of these conditions.

Dismantlable graphs are closely related with Gibbs measures, as illustrated by the following proposition, parts of which appeared in (\cite{1-brightwell}).

\begin{proposition}[{\cite{1-brightwell}}]
\label{dism-charact}
Let $\const$ be a constraint graph. Then, the following are equivalent:
\begin{enumerate}
\item $\const$ is dismantlable.
\item $\homspace$ is strongly irreducible with gap $2|\const|+1$, for every board $\board$.
\item $\homspace$ is strongly irreducible with some gap, for every board $\board$.
\item For every board $\board$ of bounded degree, there exists a Gibbs $(\board,\const,\Phi)$-spe\-ci\-fi\-ca\-tion that admits a unique n.n. Gibbs measure $\mu$.
\item For every board $\board$ of bounded degree and $\gamma > 0$, there exists a Gibbs $(\board,\const,\Phi)$-specification that satisfies exponential WSM with decay rate $\gamma$.
\end{enumerate}
\end{proposition}

The equivalence of $(1)$, $(2)$, $(3)$ and $(4)$ in Proposition \ref{dism-charact} is proven in \cite{1-brightwell}. However, in that work the concept of WSM (a priori, stronger than uniqueness) is not considered. Since WSM implies uniqueness, $(5) \implies (4)$ is trivial. In the remaining part of this section, we introduce the necessary background to prove the missing implications, for which it is sufficient to show that $(1) \implies (5)$ (see Proposition \ref{dism-highRate}). This and a subsequent proof (see Proposition \ref{UMC-highRate}) in this paper will have a similar structure to the proof that $(1) \implies (4)$ from \cite[Theorem 7.2]{1-brightwell}. However, some coupling techniques will need to be modified, plus other combinatorial ideas need to be considered.

Given a dismantlable graph $\const$, the only case where a sequence of folds reduces $\const$ to a vertex without a loop is when $\const$ is a set of isolated vertices without loops. In this case, we call $\const$ trivial (see \cite[p. 6]{1-brightwell}). If $v \in \cVert$ has a loop and there is a sequence of folds reducing $\const$ to $v$, then we call $v$ a \emph{persistent} vertex of $\const$ as in \cite{1-brightwell}.

\begin{lemma}[{\cite[Lemma 5.2]{1-brightwell}}]
\label{persistent}
Let $\const$ be a nontrivial dismantlable graph and  $v^*$ a persistent vertex of $\const$. Let $\board$ be a board, $A \Subset \bVert$, and $\omega \in \homspace$. Then there exists $\upsilon \in \homspace$ such that:
\begin{enumerate}
\item $\upsilon(x) = \omega(x)$, for every $x \in \bVert \setminus \neig_{|\const|-2}(A)$,
\item $\upsilon(x) = v^*$, for every $x \in A$, and
\item $\omega^{-1}(v^*) \subseteq \upsilon^{-1}(v^*)$.
\end{enumerate}
\end{lemma}

Now, given a constraint graph $\const$, a persistent vertex $v^*$ and $\lambda > 1$, define $\Phi_\lambda$ to be the n.n. interaction given by
\begin{equation}
\begin{array}{ccc}
\Phi_\lambda(v^*)  = -\log\lambda,	&	\mbox{ and }	&	\left.\Phi_\lambda\right|_{\cVert \setminus \{v^*\} \cup \cEdg} \equiv 0.
\end{array}
\end{equation}

We have the following lemma.

\begin{lemma}
\label{boundDismant}
Let $\const$ be dismantlable and let $\board$ be a board of bounded degree $\Delta$. Given $\lambda > 1$, consider the Gibbs $(\board,\const,\Phi_\lambda)$-spe\-ci\-fi\-ca\-tion $\pi$, a point $\omega \in \homspace$, and sets $B \subseteq A \Subset \bVert$ such that $\neig_{|\const|-2}(B) \subseteq A$. Then, for any $k \in \N$,
\begin{equation}
\pi^\omega_A\left(\left\{\alpha: \left|\left\{ y \in B: \alpha(y) \neq v^*\right\}\right| \geq k\right\}\right) \leq |\const|^{|B|\Delta^{|\const|-1}} \lambda^{-k}.
\end{equation}
\end{lemma}

\begin{proof}
Let's denote $C = \neig_{|\const|-2}(B)$. W.l.o.g., consider an arbitrary configuration $\alpha \in \cVert^A$ such that $\alpha\left.\omega\right|_{A^c} \in \homspace$ (and, in particular, such that $\pi^\omega_A(\alpha) > 0$). Denote $\omega_1 = \alpha\left.\omega\right|_{A^c}$. By Lemma \ref{persistent}, there exists $\omega_2 \in \homspace$ such that $\left.\omega_1\right|_{\bVert \setminus C} = \left.\omega_2\right|_{\bVert \setminus C}$, $\omega_2(x) = v^*$ for every $x \in B$, and $\omega_1^{-1}(v^*) \subseteq \omega_2^{-1}(v^*)$. 

Notice that $\left.\omega_1\right|_{A}\left.\omega\right|_{A^c}, \left.\omega_2\right|_{A}\left.\omega\right|_{A^c} \in \homspace$ and $\left.\omega_1\right|_{A \setminus C} = \left.\omega_2\right|_{A \setminus C}$. Now, given some $k \leq |B|$, suppose that $\alpha$ is such that $\left|\left\{ y \in B: \alpha(y) \neq  v^*\right\}\right| \geq k$. Then, by the definition of Gibbs specification and the fact that $\omega_1^{-1}(v^*) \subseteq \omega_2^{-1}(v^*)$,
\begin{equation}
\frac{\pi^\omega_A\left(\left.\omega_{2}\right|_{C} \middle\vert \left.\omega_1\right|_{A \setminus C}\right)}{\pi^\omega_A\left(\left.\omega_1\right|_{C} \middle\vert \left.\omega_1\right|_{A \setminus C}\right)} = \frac{\pi^{\omega_{1}}_C\left(\left.\omega_{2}\right|_{C}\right)}{\pi^{\omega_{1}}_C\left(\left.\omega_1\right|_{C}\right)} \geq \lambda^{k}.
\end{equation}

Therefore,
\begin{equation}
\pi^\omega_A\left(\left.\alpha\right|_{C} \middle\vert \left.\alpha\right|_{A \setminus C}\right) = \pi^\omega_A\left(\left.\omega_1\right|_{C} \middle\vert \left.\omega_1\right|_{A \setminus C}\right) \leq \lambda^{-k}.
\end{equation}

Next, by taking averages over all configurations $\beta \in \Leng_{A \setminus C}(\homspace)$ such that
\begin{equation}
\beta\left.\alpha\right|_C\left.\omega\right|_{A^c} \in \homspace,
\end{equation}
we have
\begin{equation}
\pi^\omega_A(\left.\alpha\right|_{C}) = \sum_{\beta} \pi^\omega_A\left(\left.\alpha\right|_{C} \middle\vert \beta\right) \pi^\omega_A\left(\beta\right) \leq \sum_{\beta} \lambda^{-k} \pi^\omega_A\left(\beta\right) = \lambda^{-k}.
\end{equation}

Notice that $|C| = \left|\neig_{|\const|-2}(B)\right| \leq |B|\Delta^{|\const|-1}$. In particular,
\begin{equation}
\left|\Leng_{C}\left(\homspace\right)\right| \leq |\const|^{|B|\Delta^{|\const|-1}}.
\end{equation}

Then, since $\alpha$ was arbitrary,
\begin{align}
\pi^\omega_A\left(\left\{\alpha: \left|\left\{ y \in B: \alpha(y) \neq v^*\right\}\right| \geq k\right\}\right)	&	\leq \sum_{\substack{\left.\alpha\right|_C: \alpha \in \Leng_A(\homspace),\\\left|\left\{ y \in B: \alpha(y) \neq v^*\right\}\right| \geq k}}\pi^\omega_A(\left.\alpha\right|_C)	\\
																		&	\leq |\const|^{|B|\Delta^{|\const|-1}}\lambda^{-k}.
\end{align}
\end{proof}

As mentioned before, we essentially use some coupling techniques from \cite{1-brightwell}, with slight modifications.
We will use the following theorem.

\begin{theorem}[{\cite[Theorem 1]{1-berg}}]
\label{vandenberg}
Given $\pi$ a Gibbs $(\board,\const,\Phi)$-specification, $A \Subset \bVert$ and $\omega_1,\omega_2 \in \homspace$, there exists a coupling $\left((\alpha_1(x),\alpha_2(x)), x \in A\right)$ of $\pi^{\omega_1}_A$ and $\pi^{\omega_2}_A$ (whose distribution we denote by $\mathbb{P}^{\omega_1,\omega_2}_A$), such that for each $x \in A$, $\alpha_1(x) \neq \alpha_2(x)$ if and only if there is a path of disagreement (i.e. a path $\Path$ such that $\alpha_1(y) \neq \alpha_2(y)$, for all $y \in \Path$) from $x$ to $\Sigma_{\partial A}(\omega_1,\omega_2)$, $\mathbb{P}^{\omega_1,\omega_2}_A$-almost surely.
\end{theorem}

\begin{remark}
The result in \cite[Theorem 1]{1-berg} is for MRFs, but here we state it for specifications.
\end{remark}

\begin{proposition}
\label{dism-highRate}
Let $\board$ be a board of bounded degree $\Delta$ and $\const$ be a dismantlable graph. Then, for all $\gamma > 0$, there exists $\lambda_0 = \lambda_0(\gamma,|\const|,\Delta)$ such that for every $\lambda > \lambda_0$, the Gibbs $(\board,\const,\Phi_\lambda)$-specification $\pi$ satisfies exponential WSM with decay rate $\gamma$.
\end{proposition}

\begin{proof}
Let $A \Subset \bVert$, $B \subseteq A$, $\beta \in \cVert^{B}$ and $\omega_1,\omega_2 \in \homspace$. W.l.o.g. (since $C$ can be taken arbitrarily large in the desired decay function $Ce^{-\gamma n}$), we may suppose that
\begin{equation}
\dist\left(B,\partial A\right) = n > |\const|-2.
\end{equation}

By Theorem \ref{vandenberg}, we have
\begin{align}
\left| \pi_A^{\omega_1}(\beta) - \pi_A^{\omega_2}(\beta) \right|	&	=	\left|\mathbb{P}^{\omega_1,\omega_2}_A\left(\left.\alpha_1\right|_{B} = \beta\right) - \mathbb{P}^{\omega_1,\omega_2}_A\left(\left.\alpha_2\right|_{B} = \beta\right)\right|	\\
												&	\leq	\mathbb{P}^{\omega_1,\omega_2}_A(\left.\alpha_1\right|_{B} \neq \left.\alpha_2\right|_{B})	\\
												&	\leq	\sum_{x \in B}\mathbb{P}^{\omega_1,\omega_2}_A(\alpha_1(x) \neq \alpha_2(x))			\\
												&	=	\sum_{x \in B}\mathbb{P}^{\omega_1,\omega_2}_A\left(\exists \mbox{ path of disagr. from } x \mbox{ to } \Sigma_{\partial A}(\omega_1,\omega_2)\right)	\\
												&	\leq	\sum_{x \in B}\mathbb{P}^{\omega_1,\omega_2}_A\left(\exists \mbox{ path of disagr. from } x \mbox{ to } \partial A\right)\\
												&	\leq	\sum_{x \in B}\mathbb{P}^{\omega_1,\omega_2}_A\left(\exists \mbox{ path of disagr. from } x \mbox{ to } \neig_{|\const|-2}(\partial A)\right).
\end{align}

When considering a path of disagreement $\Path$ from $x$ to $\neig_{|\const|-2}(\partial A)$, we can assume that $\neig_{|\const|-2}(\Path) \subseteq A$. Then, by Lemma \ref{boundDismant},
\begin{equation}
\pi^\omega_A\left(\left\{\alpha: \left|\left\{ y \in \Path: \alpha(y) \neq v^*\right\}\right| \geq k\right\}\right) \leq |\const|^{|\Path|\Delta^{|\const|-1}} \lambda^{-k}.
\end{equation}

In addition, we have $|\Path| \geq n-|\const|+2$ and, for every $y \in \Path$, we have $\alpha_1(y) \neq \alpha_2(y)$, so $\alpha_1(y)$ and $\alpha_2(y)$ cannot be both $v^*$ at the same time and either $\left.\alpha_1\right|_\Path$ or $\left.\alpha_1\right|_\Path$ must have $|\Path|/2$ sites different from $v^*$. In consequence,
\begin{align}
	&	\mathbb{P}^{\omega_1,\omega_2}_A\left(\exists \mbox{ path of disagr. from } x \mbox{ to } \neig_{|\const|-2}(\partial A)\right)	\\
\leq	&~	\sum_{k = n-|\const|+2}^\infty \sum_{|\Path| = k} \pi^{\omega_1}_A\left(\left\{\alpha_1: \left|\left\{ y \in \Path: \alpha_1(y) \neq v^*\right\}\right| \geq \frac{k}{2}\right\}\right)  \\
	&~	+ \sum_{k = n-|\const|+2}^\infty \sum_{|\Path| = k} \pi^{\omega_2}_A\left(\left\{\alpha_2: \left|\left\{ y \in \Path: \alpha_2(y) \neq v^*\right\}\right| \geq \frac{k}{2}\right\}\right)	\nonumber\\
\leq	&~	2\sum_{k = n-|\const|+2}^\infty \sum_{|\Path| = k} |\const|^{k\Delta^{|\const|-1}}\lambda^{-\frac{k}{2}}	\\
\leq	&~	2\sum_{k = n-|\const|+2}^\infty \Delta(\Delta-1)^k \left(\frac{|\const|^{\Delta^{|\const|-1}}}{\lambda^{1/2}}\right)^k 	\\	
=	&~	2\Delta\sum_{k = n-|\const|+2}^\infty \left(\frac{(\Delta-1)|\const|^{\Delta^{|\const|-1}}}{\lambda^{1/2}}\right)^k.
\end{align}

Finally, we have
\begin{align}
\left| \pi_A^{\omega_1}(\beta) - \pi_A^{\omega_2}(\beta) \right|	&	\leq \sum_{x \in B}\mathbb{P}^{\omega_1,\omega_2}_A\left(\exists \mbox{ path of disagr. from } x \mbox{ to } \neig_{|\const|-2}(\partial A)\right)	\\
												&	\leq 2|B|\Delta\sum_{k = n-|\const|+2}^\infty \left(\frac{(\Delta-1)|\const|^{\Delta^{|\const|-1}}}{\lambda^{1/2}}\right)^k,
\end{align}
so, in order to have exponential decay, it suffices to take
\begin{equation}
\lambda_0(|\const|,\Delta) := (\Delta-1)^2|\const|^{2\Delta^{|\const|-1}} < \lambda,
\end{equation}
and we note that any decay rate $\gamma$ is achievable by taking $\lambda$ sufficiently large.
\end{proof}

\begin{proof}[Proof of Proposition \ref{dism-charact}]
The implication $(1) \implies (5)$ follows from Proposition \ref{dism-highRate}. Since WSM implies uniqueness, we have $(5) \implies (4)$. The implications $(4) \implies (3) \implies (2) \implies (1) $ can be found in \cite[Theorem 4.1]{1-brightwell}.
\end{proof}

A priori, one would be tempted to think that the proof of Proposition \ref{dism-highRate} could give SSM instead of just WSM, since the coupling in Theorem \ref{vandenberg} involves a path of disagreement to $\Sigma_{\partial A}(\omega_1,\omega_2)$ and not just to $\partial A$, just as in the definition of SSM. One of the motivations of this paper is to illustrate that this is not the case (see the examples in Section \ref{section9}), mainly due to combinatorial obstructions. We will see that in order to have an analogous result for the SSM property, $\const$ must satisfy even stronger conditions than dismantlability, which guarantees WSM by Proposition \ref{dism-charact}. One of the main issues is that our proof required that $\dist\left(B,\partial A\right) > |\const|-2$, so $B$ cannot be arbitrarily close to $\partial A$, which is in opposition to the spirit of SSM.

Notice that, by Proposition \ref{dism-charact}, $\homspace$ is strongly irreducible for every board $\board$ if and only if $\const$ is dismantlable. Clearly, the forward direction still holds if ``strongly irreducible'' is replaced by ``TSSM,'' since TSSM implies strongly irreducible. Later, we will address the question of whether the reverse direction holds with this replacement.

For a dismantlable constraint graph $\const$ and a particular or arbitrary board $\board$, we are interested in whether or not $\homspace$ is TSSM and whether or not there exists a Gibbs $(\board,\const,\Phi)$-specification $\pi$ that satisfies exponential SSM with a decay rate that can be arbitrarily high. The following result shows that these two desired conclusions are related.

\begin{theorem}[{\cite[Theorem 5.2]{1-briceno}}]
\label{highrate}
Let $\pi$ be a Gibbs $(\Z^2,\const,\Phi)$-specification that satisfies exponential SSM with decay rate $\gamma > 4\log|\const|$. Then, $\mathrm{Hom}(\Z^2,\const)$ satisfies TSSM.
\end{theorem}

The preceding result is stated and proven for Gibbs measures satisfying SSM. However, the proof can be easily modified for specifications.

One of our main goals in this work is to look for conditions on $\homspace$ suitable for having a Gibbs specification that satisfies SSM. SSM seems to be related with TSSM, as the previous results show. In the following section, we explore some additional properties of TSSM.


\section{The unique maximal configuration property}
\label{section5}

Fix a constraint graph $\const$ and consider an arbitrary board $\board$. Given a linear order $\prec$ on the set of vertices $\cVert$, we consider the partial order (that, in a slight abuse of notation, we also denote by $\prec$) on $\cVert^\bVert$ obtained by extending coordinate-wise the linear order $\prec$ to subsets of $\bVert$, i.e. given $\alpha_1, \alpha_2 \in \cVert^A$, for some $A \subseteq \bVert$, we say that $\alpha_1 \prec \alpha_2$ iff $\alpha_1(x) \prec \alpha_2(x)$, for all $x \in A$, and  $\alpha_1 \preccurlyeq \alpha_2$ iff, for all $x \in A$, $\alpha_1(x) \prec \alpha_2(x)$ or $\alpha_1(x) = \alpha_2(x)$. In addition, if two vertices $u,v \in \cVert$ are such that $u \sim v$ and $u \prec v$, we will denote this by $u \precsim v$.

\begin{definition}
\label{umcp}
Given $g \in \N$, we say that $\homspace$ satisfies the \emph{unique maximal configuration (UMC) property with distance $g$} if there exists a linear order $\prec$ on $\cVert$ such that, for every $A \Subset \board$,
\begin{enumerate}
\item[(M1)] for every $\alpha \in \Leng_{A}(\homspace)$, there is a unique point $\omega_\alpha \in [\alpha]^\board_\const$ such that $\omega \preccurlyeq \omega_\alpha$ for every point $\omega \in [\alpha]^\board_\const$, and
\item[(M2)] for any two $\alpha_1,\alpha_2 \in \Leng_{A}(\homspace)$, $\Sigma_\bVert(\omega_{\alpha_1},\omega_{\alpha_2}) \subseteq \neig_g(\Sigma_{A}(\alpha_1,\alpha_2))$.
\end{enumerate}
\end{definition}

Notice that if $\homspace$ satisfies the UMC, then for any $\alpha \in \Leng_{A}(\homspace)$ and $\beta = \left.\alpha\right|_B$ with $B \subseteq A$, it is the case that $\omega_\alpha \preccurlyeq \omega_\beta$. This is natural, since we can see the configurations $\alpha$ and $\beta$ as ``restrictions'' to be satisfied by $\omega_\alpha$ and $\omega_\beta$, respectively. In addition, observe that condition $(M2)$ in Definition \ref{umcp} implies that $\Sigma_\bVert(\omega_\alpha,\omega_\beta) \subseteq \neig_g(A \setminus B)$. In particular, by taking $B = \emptyset$, we see that if $\homspace$ satisfies the UMC property, then there must exist a greatest element $\omega_* \in \homspace$ where $\omega \preccurlyeq \omega_*$ for every $\omega \in \homspace$ and $\Sigma_\bVert(\omega_\alpha,\omega_*) \subseteq \neig_g(A)$ for every $\alpha \in \Leng_{A}(\homspace)$.

The following proposition shows that the UMC property is related with the combinatorial properties introduced in Section \ref{section4}.

\begin{proposition}
\label{UMC-TSSM}
Suppose $\homspace$ satisfies the unique maximal configuration property with distance $g$. Then, $\homspace$ satisfies TSSM with gap $2g+1$.
\end{proposition}

\begin{proof}
Consider a pair of sites $x,y \in \bVert$ with $\dist(x,y) \geq 2g+1$, a set $S \Subset \bVert$, and configurations $\alpha \in \cVert^{\{x\}}$, $\beta \in \cVert^{\{y\}}$ and $\sigma \in \cVert^{S}$ such that $[\alpha\sigma]^\board_\const, [\sigma\beta]^\board_\const \neq \emptyset$. Take the maximal configurations $\omega_\sigma$, $\omega_{\alpha\sigma}$ and $\omega_{\sigma\beta}$. Notice that $\Sigma_\bVert(\omega_{\alpha\sigma},\omega_{\sigma}) \subseteq \neig_g(x)$ and $\Sigma_\bVert(\omega_{\sigma},\omega_{\sigma\beta}) \subseteq \neig_g(y)$. Since $\dist(x,y) \geq 2g+1$, we can conclude that $\overline{\neig_g(x)} \cap \neig_g(y) = \emptyset$ and $\neig_g(x) \cap \overline{\neig_g(y)} = \emptyset$. Therefore,
\begin{equation}
\left.\omega_{\sigma}\right|_{\neig_g(x)^c \cap \neig_g(y)^c} = \left.\omega_{\alpha\sigma}\right|_{\neig_g(x)^c \cap \neig_g(y)^c} = \left.\omega_{\sigma\beta}\right|_{\neig_g(x)^c \cap \neig_g(y)^c},
\end{equation}
and since $\homspace$ is a topological MRF, we have
\begin{equation}
\left.\omega_{\alpha\sigma}\right|_{\neig_g(x)} \left.\omega_{\sigma}\right|_{\neig_g(x)^c \cap \neig_g(y)^c} \left.\omega_{\sigma\beta}\right|_{\neig_g(y)} \in \homspace,
\end{equation}
so $[\alpha\sigma\beta]^\board_\const \neq \emptyset$. Using Lemma \ref{lemSing}, we conclude.
\end{proof}

Now, given $\cVert = \{v_1,\dots,v_k,\dots,v_{|\const|}\}$, define $v_1 \prec \cdots \prec v_k \prec \cdots \prec v_{|\const|}$ and define $\Phi^\prec_\lambda$ to be the n.n. interaction given by
\begin{align}
\Phi^\prec_\lambda(v_k)  = -k\log\lambda,	&	\mbox{ and }	\left.\Phi^\prec_\lambda\right|_{\cEdg} \equiv 0,
\end{align}
where $\lambda > 1$. We have the following lemma.

\begin{lemma}
\label{bound}
Let $\board$ be a board of bounded degree $\Delta$ and suppose that $\homspace$ satisfies the unique maximal configuration property with distance $g$. Given $\lambda > 0$, consider the Gibbs $(\board,\const,\Phi_\lambda)$-spe\-ci\-fi\-ca\-tion $\pi$, a point $\omega \in \homspace$, and sets $B \subseteq A \Subset \bVert$. Then, for any $k \in \N$,
\begin{equation}
\pi^\omega_A\left(\left\{\alpha: \left|\left\{ y \in B: \alpha(y) \prec \omega_\delta(y)\right\}\right| \geq k\right\}\right) \leq |\const|^{|B|\Delta^{g+1}} \lambda^{-k},
\end{equation}
where $\delta = \omega|_{\partial A}$.
\end{lemma}

\begin{proof}
Consider an arbitrary configuration $\alpha \in \cVert^A$ such that $\alpha\left.\omega\right|_{A^c} \in \homspace$ (and, in particular, such that $\pi^\omega_A(\alpha) > 0$). Take the set $C = A \cap \neig_{g}(B)$ and decompose its boundary $C$ into the two subsets $A \cap \partial C$ and $\partial A \cap \partial C$. Consider $D = A \cap \partial C$ and name $\eta = \left.\alpha\right|_D$. Since $\delta\eta \in \Leng_{\partial A \cup D}(\homspace)$, there exists a unique maximal configuration $\omega_{\delta\eta}$. Clearly, $\alpha(x) \leq \omega_{\delta\eta}(x)$, for every $x \in A$. Moreover, $\left.\omega_{\delta\eta}\right|_B = \left.\omega_{\delta}\right|_B$, since $\Sigma_\bVert(\omega_{\delta\eta},\omega_\delta) \subseteq \neig_g(D)$ and $\dist(B,D) > g$, so $\neig_g(D) \cap B = \emptyset$. 

Now, given some $k \leq |B|$, suppose that $\alpha$ is such that $\left|\left\{ y \in B: \alpha(y) \prec \omega_\delta(y)\right\}\right| \geq k$. Then, $\left|\left\{ y \in B: \alpha(y) \prec \omega_{\delta\eta}(y)\right\}\right| \geq k$ and, by the (topological and measure-theoretical) MRF property,
\begin{equation}
\frac{\pi^\omega_A\left(\left.\omega_{\delta\eta}\right|_{C} \middle\vert \left.\alpha\right|_{A \setminus C}\right)}{\pi^\omega_A\left(\left.\alpha\right|_{C} \middle\vert \left.\alpha\right|_{A \setminus C}\right)} = \frac{\pi^{\omega_{\delta\eta}}_C\left(\left.\omega_{\delta\eta}\right|_{C}\right)}{\pi^{\omega_{\delta\eta}}_C\left(\left.\alpha\right|_{C}\right)} \geq \lambda^{k}.
\end{equation}

Therefore,
\begin{equation}
\pi^\omega_A\left(\left.\alpha\right|_{C} \middle\vert \left.\alpha\right|_{A \setminus C}\right) \leq \lambda^{-k}.
\end{equation}

Next, by taking averages over all configurations $\beta \in \Leng_{A \setminus C}(\homspace)$ such that
\begin{equation}
\beta\left.\alpha\right|_C\left.\omega\right|_{A^c} \in \homspace,
\end{equation}
we have
\begin{equation}
\pi^\omega_A(\left.\alpha\right|_{C}) = \sum_{\beta} \pi^\omega_A\left(\left.\alpha\right|_{C} \middle\vert \beta\right) \pi^\omega_A\left(\beta\right) \leq \sum_{\beta} \lambda^{-k} \pi^\omega_A\left(\beta\right) = \lambda^{-k}.
\end{equation}

Notice that $|C| \leq \left|\neig_{g}(B)\right| \leq |B|\Delta^{g+1}$. In particular, $\left|\Leng_{C}\left(\homspace\right)\right| \leq |\const|^{|B|\Delta^{g+1}}$. Then, since $\alpha$ was arbitrary,
\begin{align}
\pi^\omega_A\left(\left\{\alpha: \left|\left\{ y \in B: \alpha(y) < \omega_\delta(y)\right\}\right| \geq k\right\}\right)	&	\leq \sum_{\substack{\left.\alpha\right|_C: \alpha \in \Leng_A(\homspace),\\\left|\left\{ y \in B: \alpha(y) < \omega_\delta(y)\right\}\right| \geq k}}\pi^\omega_A(\left.\alpha\right|_C)	\\
																					&	\leq |\const|^{|B|\Delta^{g+1}}\lambda^{-k}.
\end{align}
\end{proof}

\begin{proposition}
\label{UMC-highRate}
Let $\board$ be a board of bounded degree $\Delta$ and $\const$ be a constraint graph such that $\homspace$ satisfies the unique maximal configuration property. Then, for all $\gamma > 0$, there exists $\lambda_0 = \lambda_0(\gamma,|\const|,\Delta,g)$ such that for every $\lambda > \lambda_0$, the Gibbs $(\board,\const,\Phi_\lambda)$-specification $\pi$ satisfies exponential SSM with decay rate $\gamma$.
\end{proposition}

\begin{proof}
Let $A \Subset \bVert$, $x \in  A$, $\beta \in \cVert^{\{x\}}$ and $\omega_1,\omega_2 \in \homspace$. W.l.o.g., we may suppose that
\begin{equation}
\dist\left(x,\Sigma_{\partial A}(\omega_1,\omega_2)\right) = n > g.
\end{equation}

By Theorem \ref{vandenberg}, and similarly to the proof of Proposition \ref{UMC-highRate}, we have
\begin{align}
\left| \pi_A^{\omega_1}(\beta) - \pi_A^{\omega_2}(\beta) \right|	&	\leq	\mathbb{P}^{\omega_1,\omega_2}_A(\alpha_1(x) \neq \alpha_2(x))	\\
												&	=	\mathbb{P}^{\omega_1,\omega_2}_A\left(\exists \mbox{ path of disagr. from } x \mbox{ to } \Sigma_{\partial A}(\omega_1,\omega_2) \right)	\\
												&	\leq	\mathbb{P}^{\omega_1,\omega_2}_A\left(\exists \mbox{ path of disagr. from } x \mbox{ to } \neig_{g}(\Sigma_{\partial A}(\omega_1,\omega_2)) \right).
\end{align}

When considering a path of disagreement $\Path$ from $x$ to $\neig_{g}(\Sigma_{\partial A}(\omega_1,\omega_2))$, we can assume (by truncating if necessary) that $\Path \subseteq A \setminus \neig_{g}(\Sigma_{\partial A}(\omega_1,\omega_2))$ and $|\Path| \geq n-g$. By the UMC property, if we take $\delta_1 = \left.\omega_1\right|_{\partial A}$ and $\delta_2 = \left.\omega_2\right|_{\partial A}$, we have $\Sigma_\bVert(\omega_{\delta_1},\omega_{\delta_2}) \subseteq \neig_g(\Sigma_{\partial A}(\omega_1,\omega_2)) = \neig_g(\Sigma_{\partial A}(\delta_1,\delta_2))$, so $\left.\omega_{\delta_1}\right|_{\Path} = \left.\omega_{\delta_2}\right|_{\Path} =: \theta \in \Leng_{\Path}(\homspace)$. Since $\Path$ is a path of disagreement, for every $y \in \Path$ we have $\alpha_1(y) < \alpha_2(y) \leq \theta(y)$ or $\alpha_2(y) < \alpha_1(y) \leq \theta(y)$. In consequence, using Lemma \ref{bound} yields
\begin{align}
	&	\mathbb{P}^{\omega_1,\omega_2}_A\left(\exists \mbox{ path } \Path \mbox{ of disagr. from } x \mbox{ to } \neig_g(\Sigma_{\partial A}(\omega_1,\omega_2)) \right) \\
\leq	&~	\sum_{k = n-g}^\infty \sum_{|\Path| = k} \mathbb{P}^{\omega_1,\omega_2}_A\left(\Path \mbox{ is a path of disagr. from } x \mbox{ to } \neig_g(\Sigma_{\partial A}(\delta_1,\delta_2)) \right)	\\
\leq	&~	\sum_{k = n-g}^\infty \sum_{|\Path| = k} \pi^{\omega_1}_A\left(\left\{\alpha_1: \left|\left\{ y \in \Path: \alpha_1(y) < \theta(y)\right\}\right| \geq \frac{k}{2}\right\}\right)  \\
	&~	+ \sum_{k = n-g}^\infty \sum_{|\Path| = k} \pi^{\omega_2}_A\left(\left\{\alpha_2: \left|\left\{ y \in \Path: \alpha_2(y) < \theta(y)\right\}\right| \geq \frac{k}{2}\right\}\right)	\nonumber\\
\leq	&~	2\sum_{k = n-g}^\infty \sum_{|\Path| = k} |\const|^{k\Delta^{g+1}}\lambda^{-\frac{k}{2}} \leq 2\Delta\sum_{k = n-g}^\infty \left(\frac{(\Delta-1)|\const|^{\Delta^{g+1}}}{\lambda^{1/2}}\right)^k.
\end{align}

Then, by Lemma \ref{ssmSing}, exponential SSM holds whenever
\begin{equation}
\lambda_0(|\const|,\Delta,g) := (\Delta-1)^2|\const|^{2\Delta^{g+1}} < \lambda,
\end{equation}
and any decay rate $\gamma$ may be achieved by taking $\lambda$ large enough.

\end{proof}

Notice that here $\lambda_0$ is defined in terms of $|\const|$, $\Delta$ and $g$. In the WSM proof, $g$ implicitly depended on $|\const|$, but here the two parameters could be, a priori, virtually independent.


\section{Chordal/tree decomposable graphs}
\label{section6}

\subsection{Chordal and dismantlable graphs}

\begin{definition}
A finite simple graph $G$ is said to be \emph{chordal} if all cycles of four or more vertices have a \emph{chord}, which is an edge that is not part of the cycle but connects two vertices of the cycle.
\end{definition}

\begin{definition}
A \emph{perfect elimination ordering} in a finite simple graph $G = (V,E)$ is an ordering $v_1,\dots,v_n$ of $V$ such that $G[v_i \cup \{v_{i+1},\dots,v_n\} \cap \neig(v_i)]$ is a complete graph, for every $1\leq i \leq n = |G|$.
\end{definition}

\begin{proposition}[{\cite{1-fulkerson}}]
\label{elimination}
A finite simple graph $G$ is chordal if and only if it has a perfect elimination ordering.
\end{proposition}

\begin{definition}
A finite graph $G = (V,E)$ will be called \emph{loop-chordal} if $\Loops(G) = V$ and $G' = (V,E \setminus \{\{v,v\}: v \in V\})$ (i.e. $G'$ is a version of $G$ without loops) is chordal.
\end{definition}

\begin{proposition}
\label{loop-elimination}
Given a loop-chordal graph $G = (V,E)$, there exists an ordering $v_1,\dots,v_n$ of $V$ such that $G[v_i \cup \{v_{i+1},\dots,v_n\} \cap \neig(v_i)]$ is a loop-complete graph, for every $1\leq i \leq n = |G|$.
\end{proposition}

\begin{proof}
This follows immediately from Proposition \ref{elimination}.
\end{proof}

Proposition \ref{loop-elimination} can also be thought of as saying that a graph $G = (V,E)$ is loop-chordal iff $\Loops(G) = V$ and there exists an order $v_1 \prec \dots \prec v_n$ such that
\begin{equation}
v_i \precsim v_j \land v_i \precsim v_k \implies v_j \sim v_k.
\end{equation}

\begin{proposition}
A connected loop-chordal graph $G$ is dismantlable.
\end{proposition}

\begin{proof}
Let $v_1,\dots,v_n$ be the ordering of $V$ given by Proposition \ref{loop-elimination} and take $v \in \neig(v_1)$. Clearly, $v \in \{v_2,\dots,v_n\}$ and then we have $G[v_1 \cup \{v_2,\dots,v_n\} \cap \neig(v_1)] = G[\neig(v_1)]$ is a loop-complete graph and $v \in \neig(v_1)$. Therefore, $\neig(v_1) \subseteq \neig(v)$ and there is a fold from $G$ to $G[V \setminus \{v_1\}]$. It can be checked that $G[V \setminus \{v_1\}]$ is also loop-chordal, so we apply the same argument to $G[V \setminus \{v_1\}]$ and so on, until we end with only one vertex (with a loop).
\end{proof}

\subsection{A chordal/tree decomposition}

We say that a constraint graph $\const = (\cVert,\cEdg)$ has a \emph{chordal/tree decomposition} or is \emph{chordal/tree decomposable} if we can write $\cVert = \mathrm{C} \sqcup \mathrm{T} \sqcup \mathrm{J}$ such that:
\begin{enumerate}
\item $\const[\mathrm{C}]$ is a nonempty loop-chordal graph,
\item $\mathrm{T} = \mathrm{T}_1 \sqcup \cdots \sqcup \mathrm{T}_m$ and, for every $1 \leq j \leq m$, $\const[\mathrm{T}_j]$ is a tree such that there exist unique vertices $r_j \in \mathrm{T}_j$ (the root of $\mathrm{T}_j$) and $c^j \in \mathrm{C}$ such that $\{r_j,c^j\} \in \cEdg$,
\item $\mathrm{J} = \mathrm{J}_1 \sqcup \cdots \sqcup \mathrm{J}_n$ and, for every $1 \leq k \leq n$, $\const[\mathrm{J}_k]$ is a connected graph such that there exists a unique vertex $c^k \in \mathrm{C}$ such that $\{u,c^k\} \in \cEdg$ for every $u \in \mathrm{J}_k$, and
\item $\cEdg[\mathrm{T} : \mathrm{J}] = \cEdg[\mathrm{T}_{j_1} :\mathrm{T}_{j_2}] = \cEdg[\mathrm{J}_{k_1} :\mathrm{J}_{k_2}] = \emptyset$, for every $j_1 \neq j_2$ and $k_1 \neq k_2$.
\end{enumerate}

Notice that, for every $k$, the vertex $s_k \in \mathrm{C}$ is a safe symbol for $\const[\{s_k\} \cup \mathrm{J}_k]$.

\subsection{A natural linear order}

Given a chordal/tree decomposable constraint graph $\const$, we define a linear order $\prec$ on $\cVert$ as follows:
\begin{itemize}
\item If $w \in \mathrm{J}$ and $t \in \mathrm{T}$, then $w \prec t$.
\item If $t \in \mathrm{T}$ and $c \in \mathrm{C}$, then $t \prec c$.
\item If $w \in \mathrm{J}_k$ and $w' \in \mathrm{J}_{k'}$, for some $1 \leq k < k' \leq n$, then $w \prec w'$.
\item If $t \in \mathrm{T}_j$ and $t' \in \mathrm{T}_{j'}$, for some $1 \leq j < j' \leq m$, then $t \prec t'$.
\item Given $1 \leq k \leq n$, we fix an arbitrary order in $\mathrm{J}_{k}$.
\item Given $1 \leq j \leq m$, if $t_1,t_2 \in \mathrm{T}_j$, then:
	\begin{itemize}
	\item if $\dist(t_1,r^j) < \dist(t_2,r^j)$, then $t_1 \prec t_2$,
	\item if $\dist(t_1,r^j) > \dist(t_2,r^j)$, then $t_2 \prec t_1$, and
	\item For each $i$, we arbitrarily order the set of vertices $t$ with $\dist(t,r^j) = i$.
	\end{itemize}
\item If $c_1,c_2 \in \mathrm{C}$, then $c_1$ and $c_2$ are ordered according to Proposition \ref{loop-elimination}.
\end{itemize}

\begin{proposition}
\label{decomp-dism}
If a constraint graph $\const$ has a chordal/tree decomposition, then $\const$ is dismantlable.
\end{proposition}

\begin{proof}
W.l.o.g., suppose that $|\const| \geq 2$ (the case $|\const| = 1$ is trivial). Let $\board$ be an arbitrary board. In the following theorem (Theorem \ref{chordal-tree}), it will be proven that if $\const$ is chordal/tree decomposable, then $\homspace$ satisfies the UMC property with distance $|\const|-2$. Therefore, by Proposition \ref{UMC-TSSM}, $\homspace$ satisfies TSSM with gap $2(|\const|-2)+1$ and, in particular, $\homspace$ is strongly irreducible with gap $2|\const|+1$. Since the gap is independent of $\board$, we can apply Proposition \ref{dism-charact} to conclude that $\const$ must be dismantlable.
\end{proof}

\begin{figure}[ht]
\centering
\includegraphics[scale = 0.5]{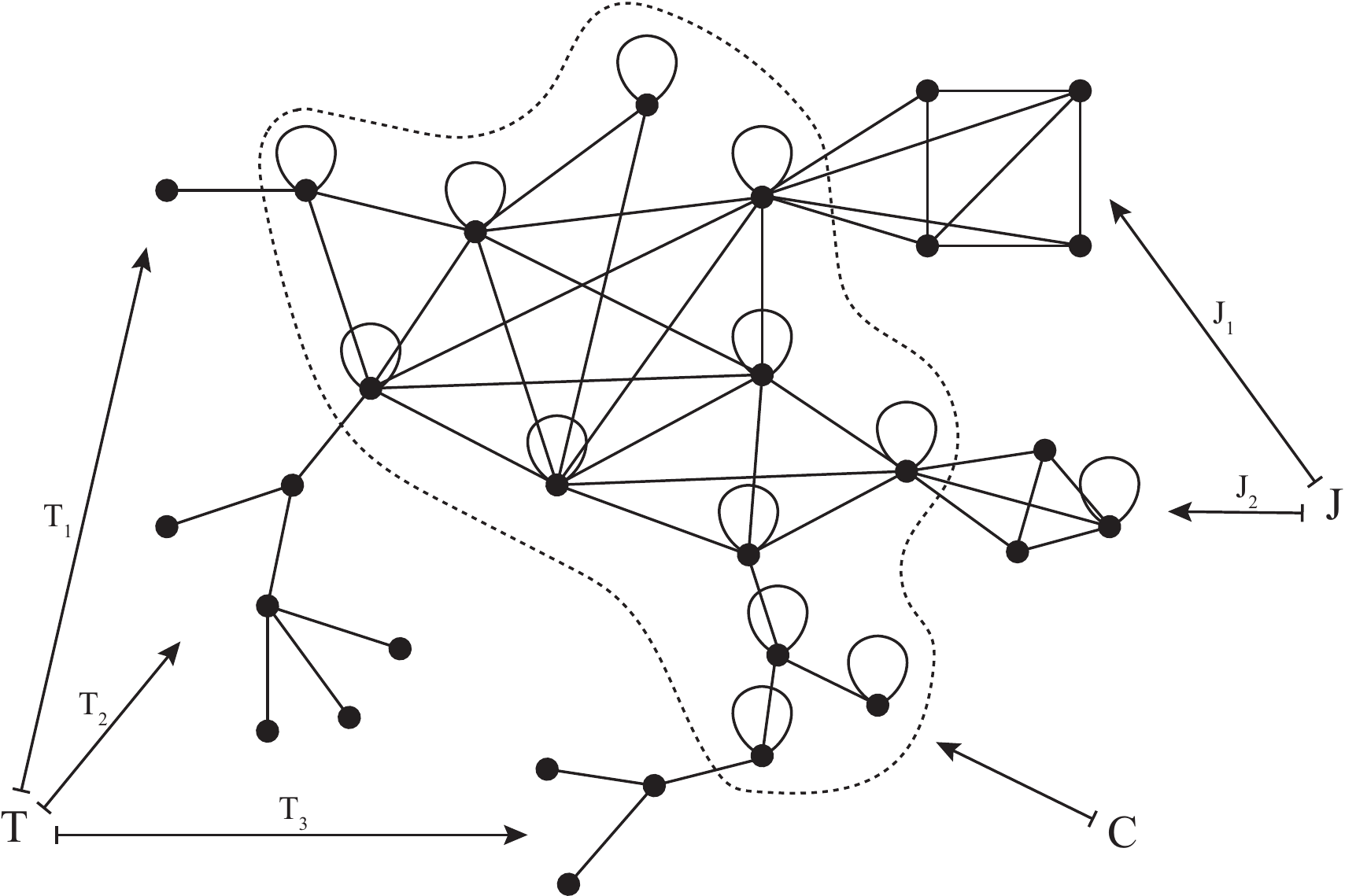} 
\caption{A chordal/tree decomposition.}
\label{decomposition}
\end{figure}

\begin{proposition}
\label{safe-decomp}
If a constraint graph $\const$ has a safe symbol, then $\const$ is chordal/tree decomposable.
\end{proposition}

\begin{proof}
This follows trivially by considering $\mathrm{C} = \{s\}$, $\mathrm{T} = \emptyset$ and $\mathrm{J} = \cVert \setminus \{s\}$, with $s$ a safe symbol for $\const$.
\end{proof}

\subsection{UMC and chordal/tree decomposable graphs}

We show that chordal/tree decomposable graphs $\const$ induce combinatorial properties on homomorphism spaces $\homspace$.

\begin{theorem}
\label{chordal-tree}
If $\const$ is a chordal/tree decomposable constraint graph, then $\homspace$ has the UMC property with distance $|\const|-2$, for any board $\board$.
\end{theorem}

Before proving Theorem \ref{chordal-tree}, we introduce some useful tools. From now on, we fix $\homspace$ and $x_1,x_2,\dots$ to be an arbitrary order of $\bVert$. We also fix the linear order $\prec$ on $\cVert$ as defined above. Given $i \in \{1,\dots,|\const|\}$, define the sets
\begin{equation}
\Dom_i := \left\{(\omega_1,\omega_2,x) \in \homspace \times \homspace \times \board: v_i = \omega_1(x) \prec \omega_2(x) \right\},
\end{equation}
and consider $\Dom(\ell) = \bigcup_{i=1}^{\ell} \Dom_i$, for $1 \leq \ell \leq |\const|$, and $\Dom := \Dom(|\const|)$. Notice that $\Dom_{|\const|} = \emptyset$. In addition, given $(\omega_1,\omega_2,x) \in \Dom$, define the set $\neig^{-}(\omega_1,x) := \{y \in \neig(x): \omega_1(y) \prec \omega_1(x)\}$ and the partition $\neig^{-}(\omega_1,x) = \neig^{-}_{\prec}(\omega_1,\omega_2,x) \sqcup \neig^{-}_{\succ}(\omega_1,\omega_2,x) \sqcup \neig^{-}_{=}(\omega_1,\omega_2,x)$, where
\begin{align}
\neig^{-}_{\prec}(\omega_1,\omega_2,x)	&	:= \{y \in \neig^{-}(\omega_1,x): \omega_1(y) \prec \omega_2(y)\},	\\
\neig^{-}_{\succ}(\omega_1,\omega_2,x)	&	:= \{y \in \neig^{-}(\omega_1,x): \omega_1(y) \succ \omega_2(y)\},	\\
\neig^{-}_{=}(\omega_1,\omega_2,x)		&	:= \{y \in \neig^{-}(\omega_1,x): \omega_1(y) = \omega_2(y)\}.
\end{align}

Let $\Pos: \Dom \to \Dom$ be the function that, given $(\omega_1,\omega_2,x) \in \Dom$, returns:
\begin{enumerate}
\item $(\omega_1,\omega_2,y)$, if $\neig^{-}_{\prec}(\omega_1,\omega_2,x) \neq \emptyset$ and $y$ is the minimal element in $\neig^{-}_{\prec}(\omega_1,\omega_2,x)$,
\item $(\omega_2,\omega_1,y)$, if $\neig^{-}_{\prec}(\omega_1,\omega_2,x) = \emptyset$, $\neig^{-}_{\succ}(\omega_1,\omega_2,x) \neq \emptyset$ and $y$ is the minimal element in $\neig^{-}_{\succ}(\omega_1,\omega_2,x)	$,
\item $(\omega_1,\omega_2,x)$, if $\neig^{-}_{\prec}(\omega_1,\omega_2,x) = \neig^{-}_{\succ}(\omega_1,\omega_2,x) = \emptyset$.
\end{enumerate}

Here the minimal elements $y$ are taken according to the previously fixed order of $\bVert$. We chose $y$ to be minimal just to have $\Pos$ well-defined; it will not be otherwise relevant. Notice that if $(\omega_1,\omega_2,x) \in \Dom_\ell$ and $\Pos(\omega_1,\omega_2,x) \neq (\omega_1,\omega_2,x)$, then $\Pos(\omega_1,\omega_2,x) \in \Dom(\ell-1)$. This implies that every element in $\Dom_1$ must be a fixed point. Moreover, for every $(\omega_1,\omega_2,x) \in \Dom$, the $(|\const|-2)$-iteration of $\Pos$ is a fixed point (though not necessarily in $D_1$), i.e. $\Pos\left(\Pos^{|\const|-2}(\omega_1,\omega_2,x)\right) = \Pos^{|\const|-2}(\omega_1,\omega_2,x)$.

We have the following lemma.

\begin{lemma}
\label{fixed-point}
Let $(\omega_1,\omega_2,x) \in D$ be such that $\Pos(\omega_1,\omega_2,x) = (\omega_1,\omega_2,x)$. Then, there exists $u \in \cVert$ such that $\omega_1(x) \prec u$ and the point $\tilde{\omega}_1$ defined as
\begin{equation}
\tilde{\omega}_1(y) =
\begin{cases}
u			&	y = x	,	\\
\omega_1(y)	&	y \neq x,
\end{cases}
\end{equation}
is globally admissible. In particular, $\omega_1 \prec \tilde{\omega}_1$.
\end{lemma}

\begin{proof}
Notice that if $(\omega_1,\omega_2,x)$ is a fixed point, $\neig^{-}(\omega_1,x) = \neig^{-}_{=}(\omega_1,\omega_2,x)$. We have two cases:

\medskip
\noindent
{\bf Case 1: $\neig^{-}(\omega_1,x) = \emptyset$.} If this is the case, then $\omega_1(y) \succcurlyeq \omega_1(x)$, for all $y \in \neig(x)$. Notice that $\omega_1(x) \prec \omega_2(x) \preccurlyeq v_{|\const|}$, so $\omega_1(x) \prec v_{|\const|}$. Then, we have three sub-cases:

\medskip
\noindent
{\bf Case 1.a: $\omega_1(x) \in \mathrm{J}_k$ for some $1 \leq k \leq n$.} Since $\{v, c^k\} \in E$ for all $v \in \mathrm{J}_k$, we can modify $\omega_1$ at $x$ in a valid way by replacing $\omega_1(x) \in \mathrm{J}_k$ with $u = c^k$.

\medskip
\noindent
{\bf Case 1.b: $\omega_1(x) \in \mathrm{T}_j$ for some $1 \leq j \leq m$.} Since $\omega_1(y) \succcurlyeq \omega_1(x)$, for all $y \in \neig(x)$, but $\omega_1(x) \in \mathrm{T}_j$ and $\mathrm{T}_j$ does not have loops, we have $\omega_1(y) \succ \omega_1(x)$, for all $y \in \neig(x)$. Call $t = \omega_1(x)$. Then, there are three possibilities: $t = r^j$, $\dist(t,r^j) = 1$ or $\dist(t,r^j) > 1$. 

If $t = r^j$, then $\omega_1(y) = c^j$ for all $y \in \neig(x)$, where $c^j \succ r^j$ is the unique vertex in $\mathrm{C}$ connected with $r^j$. Since $c^j$ must have a loop, we can replace $\omega_1(x)$ by $u = c^j \succ \omega_1(x)$ in $\omega_1$.

If $\dist(t,r^j) = 1$, then $\omega_1(y) = r^j$ for all $y \in \neig(x)$, and, similarly to the previous case, we can replace $\omega_1(x)$ by $u = c^j \succ \omega_1(x)$ in $\omega_1$.

Finally, if $\dist(t,r^j) > 1$, then $\omega_1(y) = f > t$, for all $y \in \neig(x)$, where $f \in \mathrm{T}_j$ is the parent of $t$ in the $r^j$-rooted tree $\const[\mathrm{T}_j]$. Then, since $\dist(t,r^j) > 1$, there must exist $h \in \mathrm{T}_j$ that is the parent of $f$, so we can replace $\omega_1(x)$ by $u = h$ in $\omega_1$.

\medskip
\noindent
{\bf Case 1.c: $\omega_1(x) \in \mathrm{C}$.} If this is the case, and since $\omega_1(x) \prec v_{|\const|}$, there must exist $1 \leq i < |\mathrm{C}|$ such that $\omega_1(x) = c_i$ and $\neig(c_i) \cap \{c_{i+1},\dots,c_{|\mathrm{C}|}\}$ is nonempty. Now, $\omega_1(y) \succcurlyeq \omega_1(x)$, for all $y \in \neig(x)$, so $\omega_1(y) \in \{c_i\} \cup \left(\neig(c_i) \cap \{c_{i+1},\dots,c_{|\mathrm{C}|}\}\right)$, for all $y \in \neig(x)$. Since $\const[\{c_i\} \cup \left(\neig(c_i) \cap \{c_{i+1},\dots,c_{|\mathrm{C}|}\}\right)]$ is a loop-complete graph with two or more elements, then we can replace $\omega_1(x)$ by any element $u \in \neig(c_i) \cap \{c_{i+1},\dots,c_{|\mathrm{C}|}\}$ in $\omega_1$.

\medskip
\noindent
{\bf Case 2: $\neig^{-}(\omega_1,x) \neq \emptyset$.} In this case, $\omega_1(x) \prec \omega_2(x)$ and, since 
\begin{equation}
\neig^{-}(\omega_1,x) \neq \emptyset \text{ and } \neig^{-}(\omega_1,x) = \neig^{-}_{=}(\omega_1,\omega_2,x),
\end{equation}
there must exist $y^* \in \neig(x)$ such that $\omega_1(y^*) \prec \omega_1(x)$ and $\omega_1(y^*) = \omega_2(y^*)$. Notice that in this case, $\omega_1(x)$ cannot belong to $\mathrm{T}$, because $\omega_1(x) \prec \omega_2(x)$ and both are connected to $\omega_1(y^*)$; this would imply that, for some $1 \leq j \leq m$, either (a) $\mathrm{T}_j$ does not induce a tree, or (b) more than one vertex in $\mathrm{T}_j$ is adjacent to a vertex in $\mathrm{C}$. Therefore, we can assume that $\omega_1(x)$ belongs to $\mathrm{C}$ (and therefore, since $\omega_2(x) \succ \omega_1(x)$, also $\omega_2(x)$ belongs to $\mathrm{C}$).

We are going to prove that, for every $y \in \neig(x)$, we have $\omega_1(y) \sim \omega_2(x)$, so we can replace $\omega_1(x)$ by $\omega_2(x)$ in $\omega_1$. Since $\omega_1(y) = \omega_2(y)$, for every $y \in \neig^{-}(\omega_1,x)$, we only need to prove that $\omega_1(y) \sim \omega_2(x)$, for every $y \in \neig(x)$ such that $\omega_1(y) \succcurlyeq \omega_1(x)$.

Take any $y \in \neig^{-}_{=}(\omega_1,\omega_2,x) $. Then, $\omega_2(y) = \omega_1(y) \precsim \omega_1(x)$ and $\omega_2(y) \precsim \omega_2(x)$, so $\omega_1(x) \sim \omega_2(x)$, since $\const$ is loop-chordal. Consider now an arbitrary $y \in \neig(x)$ such that $\omega_1(y) \succeq \omega_1(x)$. If $\omega_1(y) = \omega_1(x)$, we have $\omega_1(y) = \omega_1(x) \sim \omega_2(x)$, so we can assume that $\omega_1(y) \succ \omega_1(x)$. Then, $\omega_1(x) \precsim \omega_1(y)$ and $\omega_1(x) \precsim \omega_2(x)$, so $\omega_1(y) \sim \omega_2(x)$, again by loop-chordality of $\const$. Then, $\omega_1(y) \sim \omega_2(x)$, for every $y \in \neig(x)$, and we can replace $\omega_1(x)$ by $u = \omega_2(x)$ in $\omega_1$, as desired.
\end{proof}

Now we are in a good position to prove Theorem \ref{chordal-tree}.

\begin{proof}[Proof of Theorem \ref{chordal-tree}]
Fix an arbitrary set $A \Subset \bVert$ and $\alpha \in \Leng_A(\homspace)$. We proceed to prove the conditions (M1) (i.e. existence and uniqueness of a maximal point $\omega_\alpha$) and $(M2)$.

\medskip
\noindent
{\bf Condition $(M1)$.} Choose an ordering $x_1,x_2,\dots$ of $\bVert \setminus A$ and, for $n \in \N$, define $A_n := A \cup \{x_1,\dots,x_n\}$. Let $\alpha_0 := \alpha$ and suppose that, for a given $n$ and all $0 < i \leq n$, we have already constructed a sequence $\alpha_i \in \Leng_{A_i}(\homspace)$ such that $\left.\alpha_{i}\right|_{A_{i-1}} = \alpha_{i-1}$ and $\beta(x_i) \preccurlyeq \alpha_{i}(x_i)$, for any $\beta \in \Leng_{A_i}(\homspace)$ such that $\left.\beta\right|_{A_{i-1}} = \alpha_{i-1}$.

Next, look for the globally admissible configuration $\alpha_{n+1}$ such that $\left.\alpha_{n+1}\right|_{A_n} = \alpha_n$ and $\beta(x_{n+1}) \preccurlyeq \alpha_{n+1}(x_{n+1})$, for any $\beta \in \Leng_{A_{n+1}}(\homspace)$ such that $\left.\beta\right|_{A_n} = \alpha_n$. Iterating and by compactness of $\cVert^\bVert$ (with the product topology), we conclude the existence of a unique point $\hat{\omega} \in \bigcap_{n \in \N} [\alpha_n]^\board_\const$.

We claim that $\hat{\omega}$ is independent of the ordering $x_1,x_2,\dots$ of $\bVert \setminus A$.

By contradiction, suppose that given two orderings of $\bVert \setminus A$ we can obtain two different configurations $\hat{\omega}_1$ and $\hat{\omega}_2$ with the properties described above. Take $\overline{x} \in \bVert \setminus A$ such that $\hat{\omega}_1(\overline{x}) \neq \hat{\omega}_2(\overline{x})$. W.l.o.g., suppose that $\hat{\omega}_1(\overline{x}) \prec \hat{\omega}_2(\overline{x})$. Then we have that $(\hat{\omega}_1,\hat{\omega}_2,\overline{x}) \in \Dom$ and $\Pos^{|\const|-2}(\hat{\omega}_1,\hat{\omega}_2,\overline{x})$ is a fixed point for $\Pos$. W.l.o.g., suppose that $\Pos^{|\const|-2}(\hat{\omega}_1,\hat{\omega}_2,\overline{x}) = (\hat{\omega}_1,\hat{\omega}_2,\tilde{x})$, where $\tilde{x} \in \bVert \setminus A$ (note that $\tilde{x}$ is not necessarily equal to $\overline{x}$). By an application of Lemma \ref{fixed-point}, $\hat{\omega}_1(\tilde{x})$ can be replaced in a valid way by a vertex $u \in \cVert$ such that $\hat{\omega}_1(\tilde{x}) \prec u$. If we let $n$ be such that $\tilde{x} = x_n$ for the ordering corresponding to $\hat{\omega}_1$, we have a contradiction with the maximality of $\hat{\omega}_1$, since we could have chosen $u$ instead of $\hat{\omega}_1(x_n)$ in the $n$th step of the construction of $\hat{\omega}_1$.

Therefore, there exists a particular $\hat{\omega}$ common to any ordering $x_1,x_2,\dots$ of $\bVert \setminus A$. We claim that taking $\omega_\alpha = \hat{\omega}$ proves $(M1)$. In fact, suppose that there exists an $\omega \in [\alpha]^\board_\const$ and $x^* \in \bVert \setminus A$ such that $\hat{\omega}(x^*) \prec \omega(x^*)$. We can always choose an ordering of $\bVert \setminus A$ such that $x_1 = x^*$. Then, according to such ordering, $\beta \preccurlyeq \left.\hat{\omega}\right|_{A_1}$ for any $\beta \in \Leng_{A_1}(\homspace)$ such that $\left.\beta\right|_{A} = \alpha$. In particular, if we take $\beta = \alpha\left.\omega\right|_{\{x^*\}}$, we have a contradiction.

\medskip
\noindent
{\bf Condition $(M2)$.} Notice that if $(\omega_1',\omega_2',y) = \Pos(\omega_1,\omega_2,x)$, then $x = y$ or $x \sim y$. In addition, since the $|\const|$-iteration of $\Pos$ is a fixed point, if $(\omega_1',\omega_2',y) = \Pos^{|\const|-2}(\omega_1,\omega_2,x)$, then $\dist(x,y) \leq |\const|-2$. In order to prove condition $(M2)$, consider two configurations $\alpha_1,\alpha_2 \in \Leng_{A}(\homspace)$ and the set $\Sigma_\bVert(\omega_{\alpha_1},\omega_{\alpha_2})$. We want to prove that
\begin{equation}
\Sigma_\bVert(\omega_{\alpha_1},\omega_{\alpha_2}) \subseteq \neig_{|\const|-2}(\Sigma_{A}(\alpha_1,\alpha_2)).
\end{equation}

W.l.o.g., suppose that $\alpha_1 \neq \alpha_2$ and take $x \in \Sigma_\bVert(\omega_{\alpha_1},\omega_{\alpha_2}) \neq \emptyset$. It suffices to check that $\dist(x,\Sigma_{A}(\alpha_1,\alpha_2)) \leq |\const|-2$. By contradiction, suppose $\dist(x,\Sigma_{A}(\alpha_1,\alpha_2)) > |\const|-2$ and let $(\omega_{\alpha_1}',\omega_{\alpha_2}',y) = \Pos^{|\const|-2}(\omega_{\alpha_1},\omega_{\alpha_2},x)$. Notice that, by definition of $\Pos$, $y$ also belongs to $\Sigma_\bVert(\omega_{\alpha_1},\omega_{\alpha_2})$, and since $\dist(x,y) \leq |\const|$, we have $y \notin \Sigma_{A}(\alpha_1,\alpha_2)$. Then, there are two possibilities: (a) $y \in A \setminus \Sigma_{A}(\alpha_1,\alpha_2)$, or (b) $y \in \bVert \setminus A$. 

If $y \in A \setminus \Sigma_{A}(\alpha_1,\alpha_2)$, then $\omega_{\alpha_1}(y) = \alpha_1(y) = \alpha_2(y) = \omega_{\alpha_2}(y)$, and that contradicts the fact that $y \in \Sigma_\bVert(\omega_{\alpha_1},\omega_{\alpha_2})$.

If $y \in \bVert \setminus A$ and, w.l.o.g., $\omega_{\alpha_1}(y) \prec \omega_{\alpha_2}(y)$, we can apply Lemma \ref{fixed-point} to contradict the maximality of $\omega_{\alpha_1}$.
\end{proof}


\section{Summary of implications}
\label{section7}

\begin{theorem}
\label{hierarchy}

We have the following implications:

\begin{itemize}
\item Let $\const$ be a constraint graph. Then,
\begin{align*}
\const \text{ has a safe symbol}	&	\iff		\homspace \text{ is SSF } \forall \board				\\
						&	\implies	\const \text{ is chordal/tree decomposable}			\\
						&	\implies	\homspace \text{ has the UMC property } \forall \board	\\
						&	\implies	\homspace \text{ satisfies TSSM } \forall \board			\\
						&	\implies	\homspace \text{ is strongly irreducible } \forall \board	\\
						&	\iff		\const \text{ is dismantlable}.
\end{align*}

\item Let $\const$ be a constraint graph and $\board$ a fixed board. Then,
\begin{align*}
\homspace \text{ has the UMC property}	&	\implies	\homspace \text{ satisfies TSSM}		\\
								&	\implies	\homspace \text{ is strongly irreducible}.
\end{align*}

\item Let $\const$ be a constraint graph and $\board$ a fixed board with bounded degree. Then,
\begin{align*}
\homspace \text{ has the UMC property}		\implies	&	\text{For all $\gamma > 0$, there exists a Gibbs}		\\
											&	\text{$(\board,\const,\Phi)$-specification that satisfies}	\\
											&	\text{exponential SSM with decay rate $\gamma$}		\\
											&	\hspace{2.25cm}\Downarrow						\\
\const \text{ is dismantlable}				\implies	&	\text{For all $\gamma > 0$, there exists a Gibbs}		\\
											&	\text{$(\board,\const,\Phi)$-specification that satisfies}	\\
											&	\text{exponential WSM with decay rate $\gamma$}.
\end{align*}

\end{itemize}
\end{theorem}

\begin{proof}
The first chain of implications and equivalences follows from Proposition \ref{safessf}, Proposition \ref{safe-decomp}, Theorem \ref{chordal-tree}, Proposition \ref{UMC-TSSM}, Equation \ref{TSSM-irred} and Proposition \ref{dism-charact}. The second one, from Proposition \ref{UMC-TSSM} and Equation \ref{TSSM-irred}. The last one, from Proposition \ref{UMC-highRate}, Proposition \ref{dism-charact} and the fact that SSM always implies WSM. 
\end{proof}


\section{The looped tree case}
\label{section8}

A looped tree $T$ will be called trivial if $|T| = 1$ and nontrivial if $|T| \geq 2$. We proceed to define a family of graphs that will be useful in future proofs.

\begin{definition}
Given $n \in \N$, an \emph{$n$-barbell} will be the graph $B_n = \left(V(B_n),E(B_n)\right)$, where
\begin{equation}
 V(B_n) = \left\{0,1,\dots,n,n+1\right\}
\end{equation}
and
\begin{equation}
E(B_n) = \left\{\{0,0\},\{0,1\},\dots,\{n,n+1\},\{n+1,n+1\}\right\}.
\end{equation}
\end{definition}

Notice that a looped tree with a safe symbol must be an $n$-star with a loop at the central vertex, possibly along with other loops. The graph $\const_\varphi$ can be seen as a very particular case of a looped tree with a safe symbol. For more general looped trees, we have the next result.

\begin{figure}[ht]
\centering
\includegraphics[scale = 0.5]{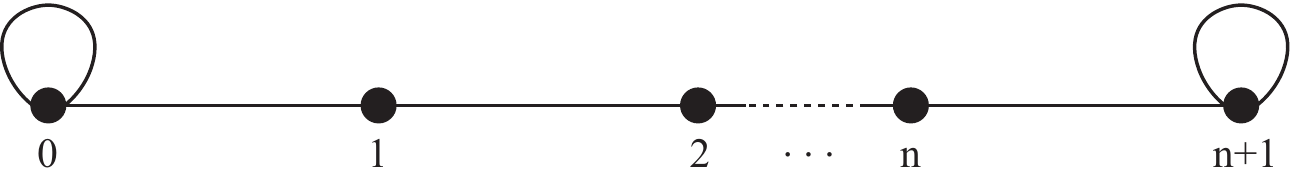} 
\caption{An $n$-barbell.}
\label{barbell}
\end{figure}

\begin{proposition}
\label{treesTFAE1}
Let $T$ be a finite nontrivial looped tree. Then, the following are equivalent:
\begin{enumerate}
\item[(1)] $T$ is chordal/tree decomposable.
\item[(2)] $T$ is dismantlable.
\item[(3)] $\Loops(T)$ is connected in $T$ and nonempty.
\end{enumerate}
\end{proposition}

\begin{proof}
We have the following implications.

\medskip
\noindent
{\bf $(1) \implies (2)$:} This follows from Theorem \ref{hierarchy}, which is for general constraint graphs.

\medskip
\noindent
{\bf $(2) \implies (3)$:} Assume that $T$ is dismantlable. First, suppose $\Loops(T) = \emptyset$. Then, in any sequence of foldings of $T$, in the next to last step, we must end with a graph consisting of just two adjacent vertices $v_{n-1}$ and $v_{n}$, without loops. However, this is a contradiction, because $\neig(u) \subsetneq \neig(v)$ and $\neig(v) \subsetneq \neig(u)$, so such graph cannot be folded into a single vertex. Therefore, $\Loops(T)$ is nonempty.

Next, suppose that $\Loops(T)$ is nonempty and not connected. Then, $T$ must have an $n$-barbell as a subgraph, for some $n \geq 1$. Therefore, in any sequence of foldings of $T$, there must have been a vertex in the $n$-barbell that was folded first. Let's call such vertex $u$ and take $v \in V$ with $\neig(u) \subseteq \neig(v)$. Then, $v$ is another vertex in the $n$-barbell or it belongs to the complement. Notice that $v$ cannot be in the $n$-barbell, because no neighbourhood of vertex in the $n$-barbell (even restricted to the barbell itself) contains the neighbourhood of another vertex in the $n$-barbell. On the other hand, $v$ cannot be in the complement of $n$-barbell, because $v$ would have to be connected to two or more vertices in the $n$-barbell ($u$ and its neighbours), and that would create a cycle in $T$. Therefore, $\Loops(T)$ is connected.

\medskip
\noindent
{\bf $(3) \implies (1)$:} Define $\mathrm{C} := \Loops(T)$. Then $\mathrm{C}$ is connected in $T$ and nonempty. Then, if we denote by $\mathrm{T}$ its complement $V \setminus \mathrm{C}$ and define $\mathrm{J} = \emptyset$, we have that $V$ can be partitioned into the three subsets $\mathrm{C} \sqcup \mathrm{T} \sqcup \mathrm {J}$, which corresponds to a chordal/tree decomposition. 
\end{proof}

\begin{corollary}
\label{treesTFAE2}
Let $T$ be a finite nontrivial looped tree. Then, the following are equivalent:
\begin{enumerate}
\item[(1)] $T$ is chordal/tree decomposable.
\item[(2)] $\mathrm{Hom}(\board,T)$ has the UMC property $\forall \board$.
\item[(3)] $\mathrm{Hom}(\board,T)$ satisfies TSSM $\forall \board$.
\item[(4)] $\mathrm{Hom}(\board,T)$ is strongly irreducible $\forall \board$.
\item[(5)] $T$ is dismantlable.
\end{enumerate}
\end{corollary}

\begin{proof}
By Theorem \ref{hierarchy}, we have $(1) \implies (2) \implies (3) \implies (4) \implies (5)$. The implication $(5) \implies (1)$ follows from Proposition \ref{treesTFAE1}.
\end{proof}

Sometimes, given a constraint graph $\const$, if a property for homomorphism spaces holds for a certain distinguished board or family of boards, then the property holds for any board $\board$. For example, this is proven in \cite{1-brightwell} for a dismantlable graph $\const$ and the strong irreducibility property, when $\board \in \{\tree_d\}_{d \in \N}$. The next result gives another example of this phenomenon.

\begin{proposition}
\label{tree-tssm}
Let $T$ be a finite looped tree. Then, the following are equivalent:
\begin{enumerate}
\item[(1)] $\mathrm{Hom}(\board,T)$ satisfies TSSM $\forall \board$.
\item[(2)] $\mathrm{Hom}(\Z^2,T)$ satisfies TSSM.	
\item[(3)] There exist Gibbs $(\Z^2,T,\Phi)$-specifications which satisfy exponential SSM with arbitrarily high decay rate.
\end{enumerate}
\end{proposition}

\begin{proof}
We have the following implications.

\medskip
\noindent
$(1) \implies (2)$: Trivial.

\medskip
\noindent
$(2) \implies (1)$: Let's suppose that $\mathrm{Hom}(\Z^2,T)$ satisfies TSSM. If $T$ is trivial, then $\mathrm{Hom}(\board,T)$ is a single point or empty, depending on whether the unique vertex in $T$ has a loop or not. In both cases, $\mathrm{Hom}(\board,T)$ satisfies TSSM $\forall \board$. If $T$ is nontrivial, then $\Loops(T)$ must be nonempty. To see this, by contradiction, first suppose that $T$ is nontrivial and $\Loops(T) = \emptyset$. Take an arbitrary vertex $u \in V(T)$ and a neighbour $v \in \neig(u)$. Notice that $\mathrm{Hom}(\Z^2,T)$ is nonempty, since the point $\omega_{u,v}$ defined as
\begin{equation}
\omega_{u,v}(x) = 
\begin{cases}
u	&	\mbox{ if } x_1 + x_2 = 0 \mod 2,	\\
v	&	\mbox{ if } x_1 + x_2 = 1 \mod 2,
\end{cases}
\end{equation}
is globally admissible. Now, if we interchange the roles of $u$ and $v$, and consider the (globally admissible) point $\omega_{v,u}$, we have $\left.\omega_{u,v}\right|_{(0,0)} = u$ and $\left.\omega_{v,u}\right|_{(2g+1,0)} = u$, for an arbitrary $g \in \N$. However, if this is the case, $\mathrm{Hom}(\Z^2,T)$ cannot be strongly irreducible with gap $g$, for any $g$ (and therefore, cannot be TSSM), because
\begin{equation}
\left[\left.\omega_{u,v}\right|_{(0,0)}\left.\omega_{v,u}\right|_{(2g+1,0)}\right]^{\Z^2}_{T} = \emptyset.
\end{equation}

A way to check this is by considering the fact that both $T$ and $\Z^2$ are bipartite graphs. Therefore, we can assume that $\Loops(T) \neq \emptyset$.

Now, suppose that $\Loops(T) \neq \emptyset$ and $\Loops(T)$ is not connected in $T$. If this is the case, $T$ must have an $n$-barbell as an induced subgraph, for some $n \geq 1$. Then, we would be able to construct configurations in $\Leng(\mathrm{Hom}(\Z^2,T))$ as shown in Figure \ref{channel}. Note that vertices in the barbell can reach each other only through the path determined by the barbell, since $T$ does not contain cycles. In Figure \ref{channel} are represented the cylinder sets $[\alpha\sigma]^{\Z^2}_{T}$ (top left), $[\sigma\beta]^{\Z^2}_{T}$ (top right) and $[\alpha\sigma\beta]^{\Z^2}_{T}$ (bottom), where:
\begin{enumerate}
\item $\alpha$ is the vertical left-hand side configuration in red, representing a sequence of nodes in the $n$-barbell that repeats $0$ but not $n+1$,
\item $\beta$ is the vertical right-hand side configuration in red, representing a sequence of nodes in the $n$-barbell that repeats $n+1$ but not $0$, and
\item $\sigma$ is the horizontal (top and bottom) configuration in black, representing loops on the vertices $0$ and $n+1$, respectively.
\end{enumerate}

It can be checked that $[\alpha\sigma]^{\Z^2}_{T}$ and $[\sigma\beta]^{\Z^2}_{T}$ are nonempty. However, the cylinder set $[\alpha\sigma\beta]^{\Z^2}_{T}$ is empty for every even separation distance between $\alpha$ and $\beta$, since $\alpha$ and $\beta$ force incompatible alternating configurations inside the ``channel'' determined by $\sigma$. Therefore, $\mathrm{Hom}(\Z^2,T)$ cannot satisfy TSSM, which is a contradiction.

We conclude that $\Loops(T)$ is nonempty and connected in $T$, and by Proposition \ref{treesTFAE1} $T$ is chordal/tree decomposable. Finally, by Proposition \ref{treesTFAE2}, we conclude that $\mathrm{Hom}(\board,T)$ satisfies TSSM $\forall \board$.

\medskip
\noindent
$(3) \implies (2)$: This follows by Theorem \ref{highrate}.

\medskip
\noindent
$(1) \implies (3)$: Since $\mathrm{Hom}(\board,T)$ satisfies TSSM $\forall \board$, $\mathrm{Hom}(\board,T)$ has the UMC property $\forall \board$ (see Corollary \ref{treesTFAE2}). In particular, $\mathrm{Hom}(\Z^2,T)$ has the UMC property. Then, (3) follows from Proposition \ref{UMC-highRate}.
\end{proof}

\begin{figure}[ht]
\centering
\includegraphics[scale = 0.825]{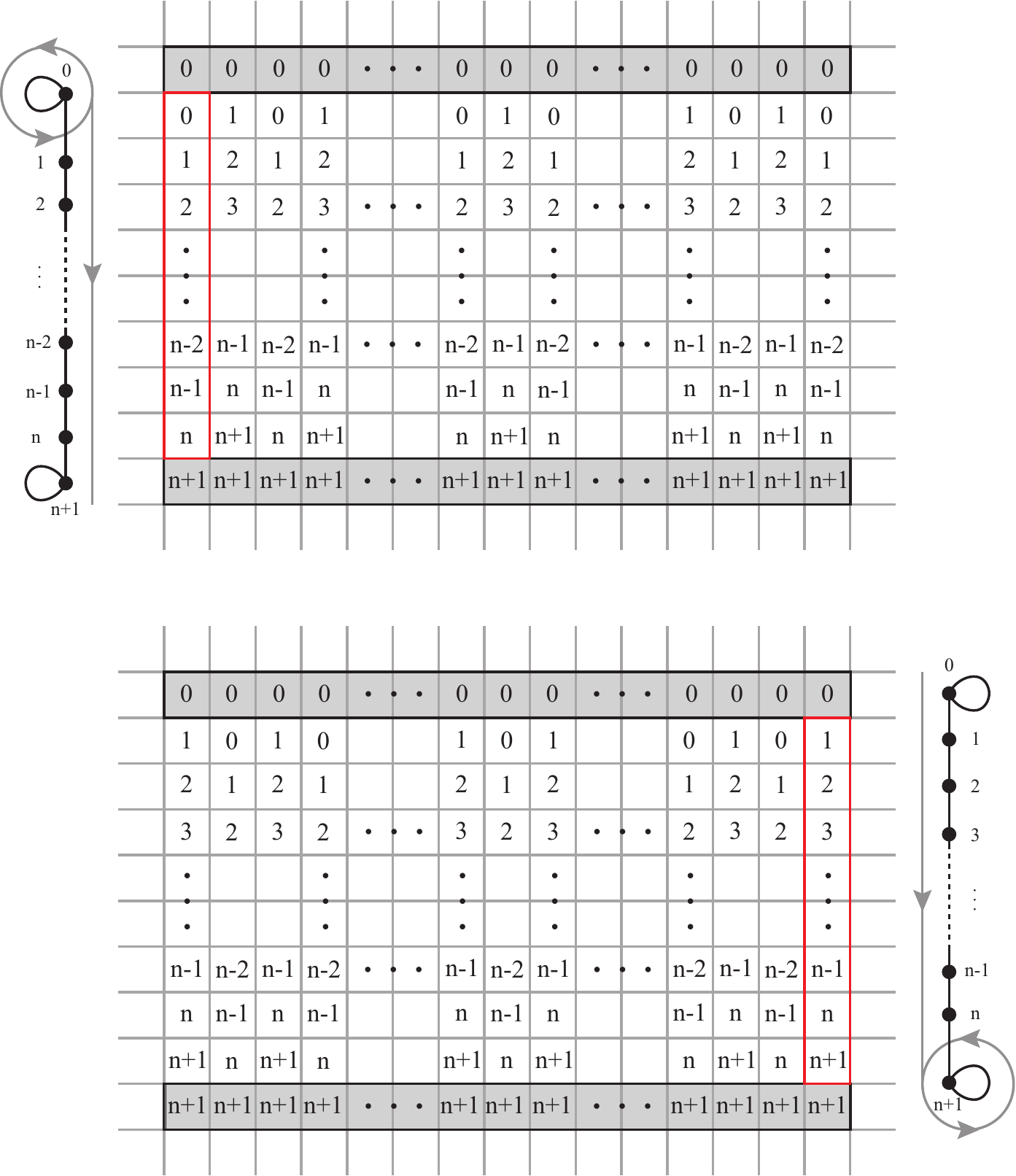} 
\caption{A ``channel'' in $\mathrm{Hom}(\Z^2,T)$ with two incompatible extremes $\alpha$ and $\beta$ (both in red).}
\label{channel}
\end{figure}


\section{Examples}
\label{section9}

\begin{proposition}
There exists a homomorphism space $\homspace$ that satisfies TSSM but not the UMC property.
\end{proposition}

\begin{proof}
The homomorphism space $\mathrm{Hom}(\Z^2,\complete_5)$ satisfies TSSM (in fact, it satisfies SSF) but not the UMC property. If $\mathrm{Hom}(\Z^2,\complete_5)$ satisfies the UMC property, then there must exist an order $\prec$ and a greatest element $\omega_* \in \mathrm{Hom}(\Z^2,\complete_5)$ according to such order (see Section \ref{section5}). Denote $V(\complete_5) = \{1,2,\dots,5\}$ and, w.l.o.g., assume that $1 \prec 2 \prec \cdots \prec 5$. Then, because of the constraints imposed by $\complete_5$, there must exist $\overline{x} \in \Z^2$ such that $\omega_*(\overline{x}) \prec \omega_*(\overline{x} + (1,0))$. Now, consider the point $\tilde{\omega}$ such that $\tilde{\omega}(x) = \omega_*(x+(1,0))$ for every $x \in \Z^2$, i.e. a shifted version of $\omega_*$ (in particular, $\tilde{\omega}$ also belongs to $\mathrm{Hom}(\Z^2,\complete_5)$). Then $\omega(\overline{x}) \prec \omega_*(\overline{x} + (1,0)) = \tilde{\omega}(\overline{x})$, which contradicts the maximality of $\omega_*$.
\end{proof}

\begin{note}
We are not aware of a homomorphism space $\homspace$ that satisfies the UMC property with $\const$ not a chordal/tree decomposable graph. 
\end{note}

\begin{figure}[ht]
\centering
\includegraphics[scale = 0.7]{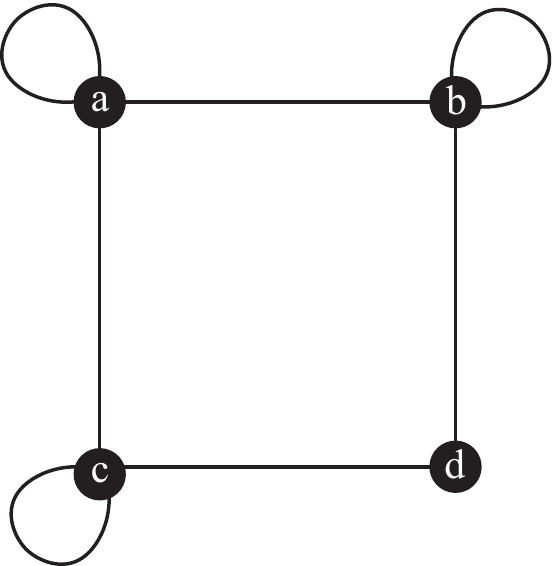} 
\caption{A dismantlable graph $\const$ such that $\mathrm{Hom}(\Z^2,\const)$ is not TSSM.}
\label{counterexample}
\end{figure}

\begin{proposition}
There exists a dismantlable graph $\const$ such that:
\begin{enumerate}
\item $\mathrm{Hom}(\Z^2,\const)$ does not satisfy TSSM.
\item There is no n.n. interaction $\Phi$ such that $(\Z^2,\const,\Phi)$ satisfies SSM.
\item For all $\gamma > 0$, there exists a n.n. interaction $\Phi$ such that the Gibbs $(\board,\const,\Phi)$-specification satisfies exponential WSM with decay rate $\gamma$.
\end{enumerate}
\end{proposition}

\begin{proof}
Consider the constraint graph $\const = (\cVert,\cEdg)$ given by
\begin{equation}
\cVert = \{a,b,c,d\} \text{ and } \cEdg = \left\{\{a,a\},\{b,b\},\{c,c\},\{a,b\},\{a,c\},\{b,d\},\{c,d\}\right\}.
\end{equation}

It is easy to check that $\const$ is dismantlable (see Figure \ref{counterexample}).

\begin{figure}[ht]
\centering
\includegraphics[scale = 0.775]{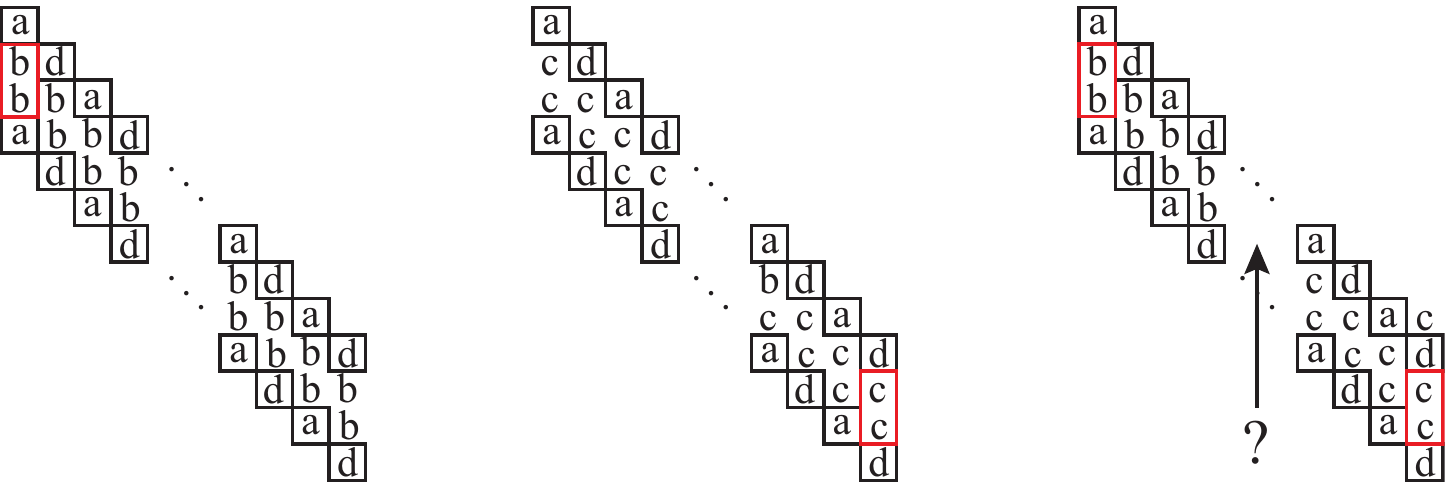} 
\caption{Two incompatible configurations $\alpha$ and $\beta$ (both in red), for a fixed configuration $\sigma$.}
\label{counterexample-conf}
\end{figure}

By Proposition \ref{dism-charact}, we know that $(3)$ holds. However, if we consider the configurations $\alpha$ and $\beta$ (the pairs $bb$ and $cc$ in red, respectively) and the fixed configuration $\sigma$ (the diagonal alternating configurations $adad\cdots$ in black) shown in Figure \ref{counterexample-conf}, we have that $[\alpha\sigma]^{\Z^2}_\const,[\sigma\beta]^{\Z^2}_\const \neq \emptyset$, but $[\alpha\sigma\beta]^{\Z^2}_\const = \emptyset$, and TSSM cannot hold. This construction works in a similar way to the construction in the proof ($(2) \implies (1)$) of Proposition \ref{tree-tssm}.

Now assume the existence of a Gibbs $(\board,\const,\Phi)$-specification $\pi$ satisfying SSM with decay function $f$. Call $A$ the shape enclosed by the two diagonals made by alternating sequence of $a$'s and $d$'s shown in Figure \ref{counterexample-conf2} (in grey), and let $x_l$ and $x_r$ be the sites  (in red) at the left and right extreme of $A$, respectively. If we denote by $\sigma$ the boundary configuration of the $a$ and $d$ symbols on $\partial A \setminus \{x_r\}$, and $\alpha_1 = b^{\{x_r\}}$ and $\alpha_2 = c^{\{x_r\}}$, it can be checked that $[\sigma\alpha_1]^\board_\const, [\sigma\alpha_2]^\board_\const \neq \emptyset$. Then, take $\omega_1 \in [\sigma\alpha_1]^\board_\const$, $\omega_2 \in [\sigma\alpha_2]^\board_\const$ and call $B = \{x_l\}$ and $\beta = b^B$.

\begin{figure}[ht]
\centering
\includegraphics[scale = 0.775]{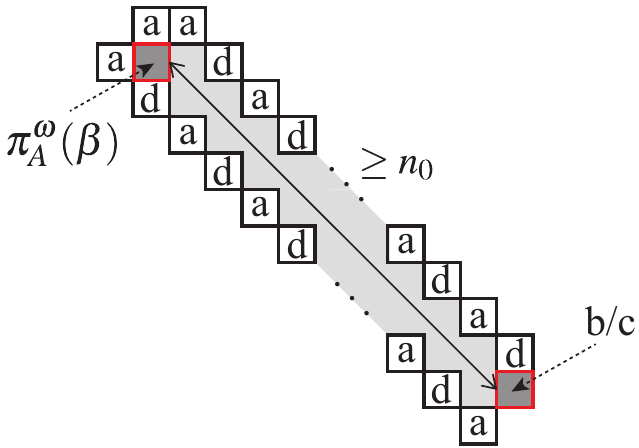} 
\caption{A shape and configurations where the SSM property fails for any Gibbs $(\Z^2,\const,\Phi)$-specification.}
\label{counterexample-conf2}
\end{figure}

Notice that, similarly as before, the symbols $b$ and $c$ force repetitions of themselves, respectively, from $x_r$ to $x_l$ along $A$. Then, we have that $\pi_A^{\omega_1}(\beta) = 1$ and $\pi_A^{\omega_2}(\beta) = 0$. Now, since we can always take an arbitrarily long set $A$, suppose that $\dist(x_l,x_r) \geq n_0$, with $n_0$ such that $f(n_0) < 1$. Therefore,
\begin{equation}
1 = |1 - 0| = \left| \pi_A^{\omega_1}(\beta) - \pi_A^{\omega_2}(\beta) \right| \leq |B|f\left(\dist(B,\Sigma_{\partial A}(\omega_1,\omega_2))\right) \leq f(n_0) < 1,
\end{equation}
which is a contradiction.
\end{proof}

\begin{proposition}
There exists a dismantlable graph $\const$ and a constant $\gamma_0 > 0$, such that:
\begin{enumerate}
\item the set of n.n. interactions $\Phi$ for which $(\Z^2,\const,\Phi)$ satisfies exponential SSM is nonempty,
\item there is no n.n. interaction $\Phi$ for which $(\Z^2,\const,\Phi)$ satisfies exponential SSM with decay rate greater than $\gamma_0$,
\item $\mathrm{Hom}(\Z^2,\const)$ satisfies SSF (in particular, $\mathrm{Hom}(\Z^2,\const)$ satisfies TSSM), and
\item for every $\gamma > 0$, there exists a n.n. interaction $\Phi$ for which $(\Z^2,\const,\Phi)$ satisfies exponential WSM with decay function $f(n) = Ce^{-\gamma n}$.
\end{enumerate}

Moreover, there exists a family $\{\const^q\}_{q \in \N}$ of dismantlable graphs with this property where $|\const^q| \to \infty$ as $q \to \infty$.
\end{proposition}

\begin{proof}
By adapting \cite[Theorem 7.3]{1-adams} to the context of Gibbs $(\Z^2,\const,\Phi)$-specifications $\pi$, we know that if $\pi$ is such that $Q(\pi) < p_{\rm c}(\Z^2)$, then $\pi$ satisfies exponential SSM, where $p_{\rm c}(\Z^2)$ denotes the critical probability for Bernoulli site percolation on $\Z^2$ and $Q(\pi)$ is defined as
\begin{equation}
Q(\pi) := \underset{\omega_1,\omega_2}{\max} \frac{1}{2} \sum_{u \in \cVert} \left|\pi^{\omega_1}_{\{0\}}(u) - \pi^{\omega_2}_{\{0\}}(u)\right|.
\end{equation}

Given $q \in \N$, consider the graph $\const^q$ as shown in Figure \ref{dismantlable-exmp}. The graph $\const^q$ consists of a complete graph $\complete_{q+1}$ and two other extra vertices $a$ and $b$ both adjacent to every vertex in the complete graph. In addition, $a$ has a loop, and $a$ and $b$ are not adjacent. Notice that $\const^q$ is dismantlable (we can fold $a$ into $b$ and then we can fold every vertex in $\complete_{q+1}$ into $a$) and $a$ is a persistent vertex for $\const^q$. Since $\const^q$ is dismantlable, by Proposition \ref{dism-charact}, for every $\gamma > 0$, there exists a n.n. interaction $\Phi$ for which the Gibbs $(\Z^2,\const,\Phi)$-specification satisfies exponential WSM with decay function $f(n) = Ce^{-\gamma n}$.

Take $\pi$ to be the uniform Gibbs specification on $\mathrm{Hom}(\Z^2,\const)$ (i.e. $\Phi \equiv 0$). Then the definition of $\pi^{\omega}_{\{0\}}(u)$ implies that, whenever $\pi^{\omega}_{\{0\}}(u) \neq 0$ for $u \in \cVert$,
\begin{align}
 \frac{1}{q + 2} \leq \pi^{\omega}_{\{0\}}(u) \leq \frac{1}{q - 1}.
\end{align}

Notice that $\pi^{\omega}_{\{0\}}(a) = 0$ if and only if $b$ appears in $\left.\omega\right|_{\partial\{0\}}$. Similarly, $\pi^{\omega}_{\{0\}}(b) = 0$ if and only if $a$ or $b$ appear in $\left.\omega\right|_{\partial\{0\}}$, and for $u \neq a,b$, $\pi^{\omega}_{\{0\}}(u) = 0$ if and only if $u$ appears in $\left.\omega\right|_{\partial\{0\}}$. Since $|\partial\{0\}| = 4$, at most 8 terms vanish in the definition of $Q(\pi)$ (4 for each $\omega_i$, $i=1,2$). Then, since $|\const^q| = q + 3$ and $q \geq 1$,
\begin{equation}
\frac{1}{2} \sum_{u \in \cVert} \left|\pi^{\omega_1}_{\{0\}}(u) - \pi^{\omega_2}_{\{0\}}(u)\right| \leq	\frac{1}{2} \left(8\frac{1}{q - 1} + (q+3)\left| \frac{1}{q - 1} - \frac{1}{q + 2}\right|\right) \leq \frac{6}{q-1}.
\end{equation}

Then, if $q > 1 + \frac{6}{p_{\rm c}(\Z^2)}$, we have that $Q(\pi) < p_{\rm c}(\Z^2)$, so $\pi$ satisfies exponential SSM. Since $p_{\rm c}(\Z^2) > 0.556$ (see \cite[Theorem 1]{4-berg}), it suffices to take $q \geq 12$. In particular, the set of n.n. interactions $\Phi$ for which the Gibbs $(\Z^2,\const^q,\Phi)$-specification satisfies exponential SSM is nonempty if $q > 12$.

\begin{figure}[ht]
\centering
\includegraphics[scale = 0.8]{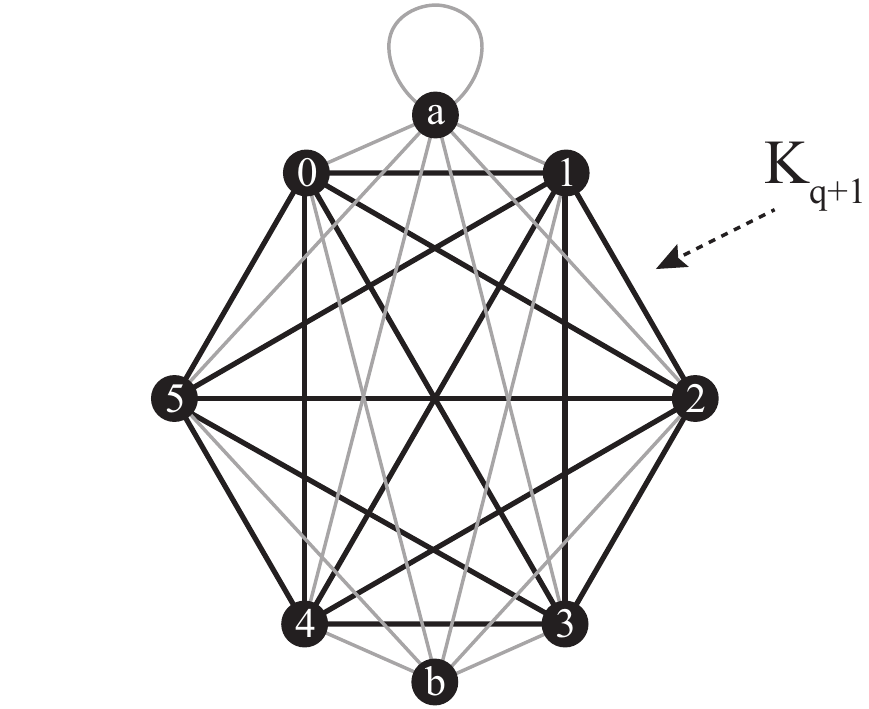} 
\caption{The graph $\const^q$, for $q = 5$.}
\label{dismantlable-exmp}
\end{figure}

Now, let $\pi$ be an arbitrary Gibbs $(\Z^2,\const^q,\Phi)$-specification that satisfies SSM with decay function $f(n) = Ce^{-\gamma n}$, for some $C$ and $\gamma$ that could depend on $q$ and $\Phi$. For now, we fix $q$ and consider an arbitrary $\Phi$.

Consider a configuration like the one shown in Figure \ref{colour-channel}. Define $\tilde{\cVert} = \cVert \backslash\{0,a,b\}$, $\tilde{\cEdg} = \cEdg[\tilde{\cVert}]$, and let $\tilde{\const}^q = \const^q[\tilde{\cVert}]$. Notice that $\tilde{\const}^q$ is isomorphic to $\complete_q$. Construct the auxiliary n.n. interaction $\tilde{\Phi}: \tilde{\cVert} \cup \tilde{\cEdg} \to (-\infty,0]$ given by $\tilde{\Phi}(u) = \Phi(u) + \Phi(u,0) + \Phi(u,b)$, for every $u \in \tilde{\cVert}$ (representing the interaction with the ``wall'' $\cdots0b0b\cdots$), and $\tilde{\Phi} \equiv \left.\Phi\right|_{\tilde{\cEdg}}$. The constrained n.n. interaction $(\tilde{\const}^q,\tilde{\Phi})$ induces a Gibbs $(\Z,\tilde{\const}^q,\tilde{\Phi})$-specification $\tilde{\pi}$ that inherits the exponential SSM property from $(\Z^2,\const^q,\Phi)$ with the same decay function $f(n) = Ce^{-\gamma n}$. It follows that there is a unique (and therefore, stationary) n.n. Gibbs measure $\mu$ for $\tilde{\pi}$, which is a Markov measure with some symmetric $q \times q$ transition matrix $P$ with zero diagonal (see \cite[Theorem 10.21]{1-georgii} and \cite{1-chandgotia}).

\begin{figure}[ht]
\centering
\includegraphics[scale = 0.65]{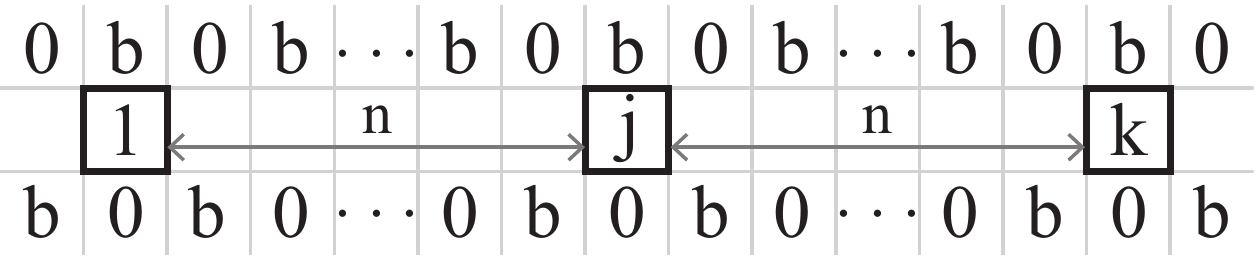} 
\caption{A Markov chain embedded in a $\Z^2$ Markov random field.}
\label{colour-channel}
\end{figure}

Let $1 = \lambda_1 \geq \lambda_2 \geq \cdots \geq \lambda_q$ be the eigenvalues of $P$. Since $\mathrm{tr}(P) = 0$, we have that $\sum_{i=1}^q \lambda_i = 0$. Let $\lambda_* = \max\{|\lambda_2|,|\lambda_q|\}$. Then, since $\lambda_1 = 1$, we have that $1 \leq \sum_{i=2}^q |\lambda_i| \leq (q-1)\lambda_*$. Therefore, $\lambda_* \geq \frac{1}{q-1}$.

Since $P$ is stochastic, $P\vec{1} = \vec{1}$ and, since $P$ is primitive, $\lambda_* < 1$ (see \cite[Section 3.2]{1-minc}). W.l.o.g., suppose that $|\lambda_2| = \lambda_*$ and let $\vec{\ell}$ be the left eigenvector associated to $\lambda_2$ (i.e. $\vec{\ell} P = \lambda_2 \vec{\ell})$. Then $\vec{\ell} \cdot \vec{1} = 0$, because $\lambda_2 \vec{\ell} \cdot \vec{1} = (\vec{\ell} \cdot P) \cdot \vec{1} = \vec{\ell} \cdot (P \cdot \vec{1}) = \vec{\ell} \cdot \vec{1}$, so $(1-\lambda_2) \vec{\ell} \cdot 1 = 0$. Then, $\vec{\ell} \in \left< \vec{e}_2 - \vec{e}_1, \vec{e}_3 - \vec{e}_1, \dots, \vec{e}_q - \vec{e}_1\right>_\R$, so we can write $\vec{\ell} = \sum_{k=2}^{q} c_k (\vec{e}_k - \vec{e}_1)$, where $\{\vec{e}_k\}_{k=1}^{q}$ denotes the canonical basis of $\R^q$ and $c_k \in \R$. We conclude that $\lambda_*^n \vec{\ell} = \vec{\ell} \cdot P^n = \sum_{k=2}^{q} c_k (\vec{e}_k - \vec{e}_1) \cdot P^n = \sum_{k=2}^{q} c_k (P^n_{k \bullet} - P^n_{1 \bullet})$, where $P^n_{i \bullet}$ is the vector given by the $i$th row of $P$.

Consider $j \in \{1,\dots,q\}$ such that $\vec{\ell}_j > 0$. Then $\lambda_*^n = \sum_{k=2}^{q} \frac{c_k}{\vec{\ell}_j} (P^n_{kj} - P^n_{1j})$ and
\begin{align}
	&	\left|P^n_{kj} - P^n_{1j}\right|																								\\	
=	& 	\left|\mu\left(\alpha(0)= j \middle\vert \alpha(-n) = k\right) 				-	\mu\left(\alpha(0)= j \middle\vert \alpha(-n) = 1\right)\right|				\\
\leq	&	\left|\mu\left(\alpha(0)= j \middle\vert \alpha(-n) = k\right) 				-	\mu\left(\alpha(0)= j \middle\vert \alpha(-n) = k, \alpha(n) = 1\right)\right|	\\
+	&	\left|\mu\left(\alpha(0)= j \middle\vert \alpha(-n) = k,\alpha(n) = 1\right)	-	\mu\left(\alpha(0)= j \middle\vert \alpha(-n) = 1, \alpha(n) = k\right)\right|	\\
+	&	\left|\mu\left(\alpha(0)= j \middle\vert \alpha(-n) = 1, \alpha(n) = k\right)	-	\mu\left(\alpha(0)= j \middle\vert \alpha(-n) = 1\right)\right|				\\
\leq	& 	3Ce^{-\gamma n},
\end{align}
by the exponential SSM property of $\tilde{\pi}$ and using that $\mu\left(\alpha(0)= j \middle\vert \alpha(-n) = k\right)$ is a weighted average $\sum_{m \in \tilde{\cVert}} \mu\left(\alpha(0)= j \middle\vert \alpha(-n) = k, \alpha(n) = m\right)\mu\left(\alpha(n) = m \middle\vert \alpha(-n) = k\right)$, along with a similar decomposition of $\mu\left(\alpha(0)= j \middle\vert \alpha(-n) = 1\right)$. Therefore, $\lambda_*^n \leq 3C(q-1)\underset{k}{\max} {\frac{|c_k|}{|\vec{\ell}_j|}}e^{-\gamma n}$. By taking logarithms and letting $n \to \infty$, we conclude that $\gamma \leq -\log \lambda_* \leq \log(q-1)$. Then, since $\Phi$ was arbitrary, there is no n.n. interaction $\Phi$ for which $(\Z^2,\const^q,\Phi)$ satisfies exponential SSM with decay rate greater than $\gamma_0 := \log(q-1)$.

Finally, it is easy to see that if $q \geq 4$, $\mathrm{Hom}(\Z^2,\const)$ satisfies SSF. Therefore, by Proposition \ref{ssftssm}, $\mathrm{Hom}(\Z^2,\const^q)$ satisfies TSSM (with gap $g = 2$).
\end{proof}

\section*{Acknowledgements}

We would like to thank Prof. Brian Marcus for his valuable contributions; his many useful observations and suggestions greatly contributed to this work. We also thank Nishant Chandgotia for discussions regarding dismantlable graphs.

\printbibliography

\end{document}